\documentclass[10pt,reqno]{amsart}
\usepackage{dsfont,amsmath,amsfonts,amscd,amssymb,graphicx,mathrsfs,eufrak,enumitem}
\usepackage[all,arc,curve,color,frame]{xy}
\usepackage[usenames]{color}

\usepackage{setspace}
\usepackage{array,multirow,booktabs,longtable}

\setlist[enumerate]{format=\normalfont}

\newcommand{\tocspace}{0.1ex}
\let\oldtocsection=\tocsection
\let\oldtocsubsection=\tocsubsection
\let\oldtocsubsubsection=\tocsubsubsection
\renewcommand{\tocsection}[3]{\hspace{0em}\oldtocsection{#1}{#2}{#3}\vspace{\tocspace}}
\renewcommand{\tocsubsection}[3]{ \hspace{1em} \oldtocsubsection{#1}{\small{#2}}{\small{#3}}\vspace{\tocspace} }
\renewcommand{\tocsubsubsection}[3]{\hspace{2em}\oldtocsubsubsection{#1}{\small{#2}}{\small{#3}}}

\newcommand{\marginparstretch}{0.6}
\let\oldmarginpar\marginpar
\renewcommand\marginpar[1]{\-\oldmarginpar[\framebox{\setstretch{\marginparstretch}\begin{minipage}{\marginparwidth}{\raggedleft\tiny #1}\end{minipage}}]{\framebox{\setstretch{\marginparstretch}\begin{minipage}{\marginparwidth}{\raggedright\tiny #1}\end{minipage}}}}

\usepackage[colorlinks]{hyperref}
\usepackage{tikz,mathrsfs}
\usetikzlibrary{arrows,decorations.pathmorphing,decorations.pathreplacing,positioning,shapes.geometric,shapes.misc,decorations.markings,decorations.fractals,calc,patterns}

\tikzset{
        cvertex/.style={circle,draw=black,inner sep=1pt,outer sep=3pt},
        vertex/.style={circle,fill=black,inner sep=1pt,outer sep=3pt},
        star/.style={circle,fill=yellow,inner sep=0.75pt,outer sep=0.75pt},
        tvertex/.style={inner sep=1pt,font=\scriptsize},
        gap/.style={inner sep=0.5pt,fill=white}}

\tikzstyle{mybox} = [draw=black, fill=blue!10, very thick,
    rectangle, rounded corners, inner sep=10pt, inner ysep=20pt]
\tikzstyle{boxtitle} =[fill=blue!50, text=white,rectangle,rounded corners]

\newtheorem{thm}{Theorem}[section]
\newtheorem{prop}[thm]{Proposition}
\newtheorem{lemma}[thm]{Lemma}
\newtheorem{defin}[thm]{Definition}
\newtheorem{cor}[thm]{Corollary}

\newtheorem{assumption}[thm]{Assumption}

\theoremstyle{definition} 

\newtheorem{example}[thm]{Example}
\newtheorem{setup}[thm]{Setup}

\newtheorem{remark}[thm]{Remark}

\newtheorem{notation}[thm]{Notation}

\numberwithin{equation}{section}

\newcounter{tempenum}

\newcommand\citetype[1]{}

\newcommand{\m}{\mathfrak{m}}
\newcommand{\n}{\mathfrak{n}}
\newcommand{\p}{\mathfrak{p}}
\newcommand{\q}{\mathfrak{q}}

\def\op{\mathop{\rm op}\nolimits}

\def\CM{\mathop{\rm CM}\nolimits}
\def\MCM{\mathop{\rm MCM}\nolimits}
\def\uCM{\mathop{\underline{\rm CM}}\nolimits}

\def\depth{\mathop{\rm depth}\nolimits}

\def\hgt{\mathop{\rm ht}\nolimits}
\def\mod{\mathop{\rm mod}\nolimits}

\def\coh{\mathop{\rm coh}\nolimits}
\def\Qcoh{\mathop{\rm Qcoh}\nolimits}
\def\Mod{\mathop{\rm Mod}\nolimits}

\def\pd{\mathop{\rm pd}\nolimits}
\def\id{\mathop{\rm inj.dim}\nolimits}
\def\uHom{\mathop{\underline{\rm Hom}}\nolimits}
\def\uEnd{\mathop{\underline{\rm End}}\nolimits}
\def\Hom{\mathop{\rm Hom}\nolimits}

\def\RHom{\mathop{\rm {\bf R}Hom}\nolimits}
\def\End{\mathop{\rm End}\nolimits}
\def\Ext{\mathop{\rm Ext}\nolimits}

\def\add{\mathop{\rm add}\nolimits}

\def\Ker{\mathop{\rm Ker}\nolimits}
\def\ker{\mathop{\rm ker}\nolimits}

\def\Im{\mathop{\rm Im}\nolimits}
\def\Sing{\mathop{\rm Sing}\nolimits}
\def\Supp{\mathop{\rm Supp}\nolimits}
\def\Ass{\mathop{\rm Ass}\nolimits}

\def\Spec{\mathop{\rm Spec}\nolimits}

\def\Max{\mathop{\rm Max}\nolimits}

\def\Perf{\mathop{\rm{Perf}}\nolimits}

\def\D{\mathop{\rm{D}^{}}\nolimits}

\def\Db{\mathop{\rm{D}^b}\nolimits}

\def\Id{\mathop{\rm{Id}}\nolimits}

\newcommand{\K}{\mathop{{}_{}\mathds{k}}\nolimits}

\newcommand{\con}{\mathrm{con}}

\newcommand{\CA}{\mathrm{A}_{\con}}

\newcommand{\AB}{\mathrm{A}}

\def\RA{\mathop{\rm RA}\nolimits}

\def\redu{\mathop{\rm red}\nolimits}

\def\uotimes{\mathop{\underline{\otimes}}\nolimits}

\def\Rf{{\rm\bf R}f}

\def\RHom{{\rm{\bf R}Hom}}

\def\sEnd{\mathcal{E}nd}
\def\RsHom{{\bf R}\mathcal{H}om}
\newcommand\RDerived[1]{{\rm\bf R}{#1}}

\newcommand\Rfi[1]{{\rm\bf R}^{#1}f}

\newcommand\Lotimes[1]{\mathop{\otimes^{\bf L}_{#1}}\nolimits}

\newcommand\art{\mathsf{Art}}
\newcommand\proart{\mathsf{pArt}}

\newcommand\Sets{\mathsf{Sets}}

\newcommand\Def{\mathcal{D}ef}

\newcommand\DG{\mathsf{DG}}

\newcommand{\cA}{\mathcal{A}}
\newcommand{\cB}{\mathcal{B}}
\newcommand{\cC}{\mathcal{C}}
\newcommand{\cD}{\mathcal{D}}
\newcommand{\cE}{\mathcal{E}}
\newcommand{\cF}{\mathcal{F}}

\newcommand{\cI}{\mathcal{I}}

\newcommand{\cK}{\mathcal{K}}

\newcommand{\cM}{\mathcal{M}}

\newcommand{\cO}{\mathcal{O}}
\newcommand{\cP}{\mathcal{P}}
\newcommand{\cQ}{\mathcal{Q}}

\newcommand{\cS}{\mathcal{S}}
\newcommand{\cT}{\mathcal{T}}

\newcommand{\cV}{\mathcal{V}}

\newcommand{\Per}{{}^{0}\mathcal{P}er}
\newcommand{\PerOne}{{}^{-1}\mathcal{P}er}

\newcommand\CAz{\mathrm{A}_{\con,z}}

\newcommand\locGen{\cV}
\newcommand\vlocGen{\cP}
\newcommand\locGenZar{\cV|_U}

\newcommand\sTiltAlg{\cA}
\newcommand\vTiltAlg{\cB}
\newcommand\sIdeal{\cI}
\newcommand\sDefAlg{\cD}
\newcommand\nonIso{Z}
\newcommand\twistGen{{\sf Twist}}
\newcommand\twist{\twistGen_X}

\newcommand\twistOb{\cE}

\newcommand\funcompose{\circ}
\newcommand\placeholder{-}

\newlength\tempWidth

\begin{document}
\title{\textsc{Noncommutative Enhancements of Contractions}}
\author{Will Donovan}
\address{Will Donovan, Yau Mathematical Sciences Center, Tsinghua University, Haidian District, Beijing 100084, China.}
\email{donovan@mail.tsinghua.edu.cn}
\author{Michael Wemyss}
\address{Michael Wemyss, The Mathematics and Statistics Building, University of Glasgow, University Place, Glasgow, G12 8QQ, UK.}
\email{michael.wemyss@glasgow.ac.uk}
\begin{abstract}

Given a contraction of a variety $X$ to a base $Y$, we enhance the locus in~$Y$ over which the contraction is not an isomorphism with a certain sheaf of noncommutative rings $\sDefAlg$, under mild assumptions which hold in the case of (1)~crepant partial resolutions admitting a tilting bundle with trivial summand, and (2)~all contractions with fibre dimension at most one.  In all cases, this produces a  global invariant. In the crepant setting, we then apply this to study derived autoequivalences of~$X$.  We work generally, dropping many of the usual restrictions, and so both extend and unify existing approaches. In full generality we construct a new endofunctor of the derived category of~$X$ by twisting over~$\sDefAlg$, and then, under appropriate restrictions on singularities, give conditions for when it is an autoequivalence. We show that these conditions hold automatically when the non-isomorphism locus in $Y$ has codimension~$3$ or~more, which covers determinantal flops, and we also control the conditions when the non-isomorphism locus has codimension~$2$, which covers $3$-fold divisor-to-curve contractions.

\end{abstract}
\thanks{The first author was supported by World Premier International Research Center Initiative (WPI), MEXT, Japan, and JSPS KAKENHI Grant Number~JP16K17561. The second author was supported by EPSRC grant~EP/K021400/1.}
\maketitle
\parindent 20pt
\parskip 0pt

\tableofcontents

\section{Introduction}
It has become increasingly clear that various aspects of birational geometry should be enhanced into a mildly noncommutative setting. One example of this, namely associating noncommutative algebras to certain contractions, has recently yielded many new results, including: algorithmic ways to relate minimal models based on cluster theory \cite{HomMMP}, new invariants for flips and flops \cite{DW1} linked to Gopakumar--Vafa invariants \cite{Toda GV}, the braiding of flop functors \cite{DW3}, noncommutative versions of curve counting \cite{Toda virtual}, a~full conjectural analytic classification of $3$-fold flops \cite{DW1, HT}, and, in addition, the first new examples of $3$-fold flops since 1983 \cite{BW, Karmazyn}.

However, a major technical and philosophical drawback of many of these constructions is that they are local in nature, or apply only to contractions of curves. In this paper, we work in a much more general setting. We take a contraction $f \colon X \to Y$ satisfying mild characteristic-free assumptions, and enhance the non-isomorphism locus $\nonIso$ in~$Y$ with  a certain sheaf of noncommutative rings.  Amongst other things, we use this sheaf to give a new class of derived autoequivalences.

In this paper we accomplish the following.
\begin{itemize}
\item\label{Q1 intro}  In an axiomatic framework sketched in \S\ref{New inv section}, we construct a noncommutative ringed space $(\nonIso,\sDefAlg)$ where $\nonIso$ is the non-isomorphism locus in~$Y$. In $3$-folds this applies to flopping, flipping, and divisor-to-curve contractions, but also works in arbitrary dimension, and for contractions with higher-dimensional fibres.

\item\label{Q2 intro} In a crepant setting given in \S\ref{section construct twist}, using the sheaf of noncommutative rings  $\sDefAlg$ above, we construct an endofunctor $\twist$ of $\D(\Qcoh X)$  associated to the contraction $f$, and establish criteria for it to be an autoequivalence.
\end{itemize}

The main benefit of this approach is its generality: we show in~\ref{assumptions do hold}\eqref{assumptions do hold 1} that the axiomatic framework applies when $f$ is a crepant resolution of a Gorenstein $d$-fold $Y$, with mild restrictions which hold in many known settings, including those of Haiman~\cite{Haiman}, Bezrukavnikov--Kaledin~\cite{BK}, Procesi bundles~\cite{Losev}, Toda--Uehara~\cite{Toda Uehara}, Springer resolutions of determinantal varieties~\cite{BLV}, and all known $3$-fold crepant resolutions including derived McKay correspondence~\cite{BKR}.  We note in~\ref{Craw--Ishii} that it covers all $3$-fold projective crepant resolutions, if the Craw--Ishii conjecture holds. Further we show in~\ref{assumptions do hold}\eqref{assumptions do hold 2} that our framework applies to all contractions with fibres of dimension at most one.

\medskip

As the construction of $\sDefAlg$ is the most subtle part of the paper, we first briefly outline the local-to-global problems that arise in one specific example.    Consider the following crepant divisor-to-curve contraction, given explicitly in \ref{D4 to A1 ex}.  $Y$ is singular along a one-dimensional locus $\nonIso$, and above every point in $\nonIso$ is an irreducible curve.
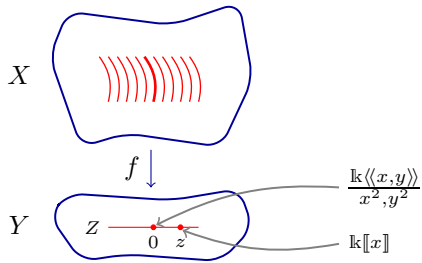
\[
\begin{array}{c}
\begin{tikzpicture}
\node at (0,0) {\begin{tikzpicture}[scale=0.6]
\coordinate (T) at (1.9,2);
\coordinate (B) at (2.1,1);
\draw[red,line width=1pt] (T) to [bend left=25] (B);
\foreach \y in {0.2,0.4,...,1}{ 
\draw[line width=0.5pt,red] ($(T)+(\y,0)$) to [bend left=25] ($(B)+(\y,0)$);
\draw[line width=0.5pt,red] ($(T)+(-\y,0)$) to [bend left=25] ($(B)+(-\y,0)$);;}
\draw[color=blue!60!black,rounded corners=6pt,line width=0.75pt] (0.5,0) -- (1.5,0.3)-- (3.6,0) -- (4.3,1.5)-- (4,3.2) -- (2.5,2.7) -- (0.2,3) -- (-0.2,2)-- cycle;
\end{tikzpicture}};
\node at (0,-2) {\begin{tikzpicture}[scale=0.6]
\draw [red] (1.1,0.75) -- (3.1,0.75);
\draw[color=blue!60!black,rounded corners=6pt,line width=0.75pt] (0.5,0) -- (1.5,0.15)-- (3.6,0) -- (4.3,0.75)-- (4,1.6) -- (2.5,1.35) -- (0.2,1.5) -- (-0.2,1)-- cycle;
\filldraw[red] (2.1,0.75) circle (1.5pt);
\filldraw[red] (2.7,0.75) circle (1.5pt);
\node at (2.1,0.43) {$\scriptstyle 0$};
\node at (2.7,0.4) {$\scriptstyle z$};
\node at (0.75,0.75) {${\scriptstyle \nonIso}$};
\end{tikzpicture}};
\draw[->, color=blue!60!black] (0,-1) -- (0,-1.5);
\node at (-1.75,0) {$X$};
\node at (-1.75,-2) {$Y$};
\node at (-0.25,-1.25) {$f$};
\draw[->,line width=0.75,black!50,bend right=15] (2.5,-1.5) to (0.07,-2);
\draw[->,line width=0.75,black!50,bend left=15] (2.5, -2.25) to (0.42,-2.07);
\node at (3.1,-1.5) {$\frac{\K\langle\!\langle x,y\rangle\!\rangle}{x^2,y^2}$};
\node at (2.9,-2.25) {$\scriptstyle \K[\![ x]\!]$};
\end{tikzpicture}
\end{array}
\]

At every closed point $z\in\nonIso$, \cite{DW2} constructs a noncommutative deformation algebra $\CAz$, which induces a universal sheaf $\twistOb_z$ on the formal fibre above $z$. In the above example, at every closed point of $\nonIso$ away from the origin, $\CAz$ is isomorphic to $\K[\![ x]\!]$, and at the origin it has the form shown above.  The question is whether there exists a global sheaf of algebras $\sDefAlg$ on $\nonIso$ which specialises, complete locally, to the algebras $\CAz$.  Similarly one can ask whether the universal sheaves $\twistOb_z$, as $z$ varies over $\nonIso$, can be glued into a single coherent sheaf $\twistOb$ on $X$.

The key insight in this paper is that this can be done, but not on the nose; we construct a global $\sDefAlg$ that recovers the algebras $\CAz$ up morita equivalence, and in~\eqref{equation universal sheaf} a global $\twistOb$ that recovers the $\twistOb_z$ up to additive equivalence.  In the above example this means that $\sDefAlg$ completed at a point $z$ away from the origin is the ring of $2\times 2$ matrices over $\K[\![ x]\!]$, instead of $\K[\![ x]\!]$ itself.  This is rather harmless, since to extract local invariants, we pass to the basic algebra~\cite[\S 3]{DW1}, and for construction of derived autoequivalences, passing through morita equivalences is a mild and necessary procedure.

\subsection{Construction of Invariant $(Z,\sDefAlg)$ and Global $\twistOb$}\label{New inv section}
We sketch this briefly, before describing our main results and applications. The construction does not use deformation theory, or any restrictions on singularities or fibre dimension.  Full details are given in~\S\ref{section global thickening}.

By a contraction, we mean a projective birational map $f\colon X\to Y$ between $d$-dimensional normal varieties over an algebraically closed field $\K$,  satisfying $\Rf_*\cO_X=\cO_{Y}$, where we assume that $Y$ is quasi-projective. Write $\nonIso$ for the locus of points of $Y$ above which $f$ is not an isomorphism.  

We then assume that there is a vector bundle $\vlocGen$ on $X$ containing $\cO_X$ as  a summand, such that
\begin{enumerate}
\item The natural map $f_*\sEnd_X(\vlocGen)\to \sEnd_Y(f_*\vlocGen)$ is an isomorphism.
\item The functor
\begin{equation}
\RDerived{f}_* \RsHom_X(\vlocGen,\placeholder) \colon \Db(\coh X)\to\Db(\mod \vTiltAlg)\label{Db equiv intro}
\end{equation}
is an equivalence, where $\vTiltAlg:=f_*\sEnd_X(\vlocGen)$ is considered as a sheaf of $\cO_Y$-algebras.
\end{enumerate}

Given this setup, we write $\cQ:=f_*\vlocGen$, so that $\vTiltAlg\cong\sEnd_Y(\cQ)$.  We then define a subsheaf $\sIdeal$ of $\vTiltAlg$ consisting of local sections that at each stalk factor through a finitely generated projective module, in \ref{defin sheaf of ideals}. We show in \ref{is sheaf of ideals} that $\sIdeal$ is naturally a subsheaf of two-sided ideals of $\vTiltAlg$.  

Thus we may also view $\sIdeal$ as a subsheaf of $\vTiltAlg^{\op}$, and define $\sDefAlg:=\vTiltAlg^{\op}/\sIdeal$, which we call the \emph{sheafy contraction algebra}.  For most purposes, taking the opposite ring structure can be ignored.  In \ref{3 ways to def} we give two alternative descriptions of $\sIdeal$, which are important since they allow us to control local sections of $\sDefAlg$, to such an extent that we can prove the following.

\begin{thm}[=\ref{global cont thm}, Global Contraction Theorem]
$\Supp_Y\sDefAlg=\nonIso$.
\end{thm}

It follows that $\sDefAlg$ is naturally a sheaf of algebras on the locus $\nonIso$, and thus we can view the ringed space $(\nonIso,\sDefAlg)$ as a noncommutative enhancement of $\nonIso$.

\begin{remark}When the locus $\nonIso$ can be realized as a moduli spaces $\cM$ of stable sheaves, Toda \cite{Toda thickening} constructs, under restrictions on $\operatorname{dim} \nonIso$, another noncommutative enhancement of $\nonIso$ which is different to ours, since it is unusual for the stalks of $\sDefAlg$ to be local rings; see also \ref{motivate hyper assump} and \ref{not Toda}.\end{remark}

The noncommutative sheaf of rings $\sDefAlg$ constructed is naturally a sheaf on the base $Y$ of the contraction $f\colon X\to Y$.  We next lift this to give an object on~$X$, by observing that by construction $\sDefAlg$ is a factor of~$\vTiltAlg^{\op}$, so it also carries the natural structure of a $\vTiltAlg$-bimodule. In particular, we can view $\sDefAlg$ as an object of~$\Db(\mod\vTiltAlg^{\op})$, and hence, across a dual version of \eqref{Db equiv intro}, it gives an object
\begin{equation}\label{equation universal sheaf}
\twistOb:=f^{-1} \sDefAlg \Lotimes{f^{-1} \vTiltAlg^{\op}} \vlocGen^*
\end{equation}
of $\Db(\coh X)$. In general $\twistOb$ need not be a sheaf. In our most general setup, we prove the following.

\begin{prop}[=\ref{derived chase}]
$\Rf_*\twistOb=0$. In particular, $\Supp_X\twistOb$ lies in the exceptional locus.
\end{prop}

The case of one-dimensional fibres  is particularly pleasant, and a typical example is sketched below.
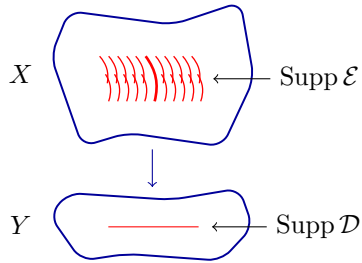
\[
\begin{array}{c}
\begin{tikzpicture}
\node at (0,0) {\begin{tikzpicture}[scale=0.6]
\coordinate (T) at (1.9,2);
\coordinate (TM) at (2.12-0.05,1.5-0.1);
\coordinate (BM) at (2.12-0.09,1.5+0.1);
\coordinate (B) at (2.1,1);
\draw[red,line width=1pt] (T) to [bend left=25] (B);
\foreach \y in {0.2,0.4,...,1}{ 
\draw[line width=0.5pt,red] ($(T)+(\y,0)+(0.02,0)$) to [bend left=25] ($(TM)+(\y,0)+(0.02,0)$);
\draw[line width=0.5pt,red] ($(BM)+(\y,0)+(0.02,0)$) to [bend left=25] ($(B)+(\y,0)+(0.02,0)$);
\draw[line width=0.5pt,red] ($(T)+(-\y,0)+(-0.02,0)$) to [bend left=25] ($(TM)+(-\y,0)+(-0.02,0)$);
\draw[line width=0.5pt,red] ($(BM)+(-\y,0)+(-0.02,0)$) to [bend left=25] ($(B)+(-\y,0)+(-0.02,0)$);;}
\draw[color=blue!60!black,rounded corners=5pt,line width=0.75pt] (0.5,0) -- (1.5,0.3)-- (3.6,0) -- (4.3,1.5)-- (4,3.2) -- (2.5,2.7) -- (0.2,3) -- (-0.2,2)-- cycle;
\node (twistOb label) at (5.75,1.5) {$\Supp \twistOb$};
\draw[->] (twistOb label) -- (3.4,1.5);
\end{tikzpicture}};
\node at (0,-2) {\begin{tikzpicture}[scale=0.6]
\draw [red] (1.1,0.75) -- (3.1,0.75);
\draw[color=blue!60!black,rounded corners=5pt,line width=0.75pt] (0.5,0) -- (1.5,0.15)-- (3.6,0) -- (4.3,0.75)-- (4,1.6) -- (2.5,1.35) -- (0.2,1.5) -- (-0.2,1)-- cycle;
\node (sDefAlg label) at (5.75,0.75) {$\Supp \sDefAlg$};
\draw[->] (sDefAlg label) -- (3.4,0.75);
\end{tikzpicture}};
\draw[->, color=blue!60!black] (-0.75,-1) -- (-0.75,-1.5);
\node at (-2.5,0) {$X$};
\node at (-2.5,-2) {$Y$};
\end{tikzpicture}
\end{array}
\]
The following are our main results in this setting.  In particular, this globalises the local noncommutative deformation theory as studied in \cite{DW1, DW2, BB, Kawamata}.

\begin{thm}\label{main E results into}
Suppose that $f\colon X\to Y$ is a contraction where the fibres have dimension at most one. Then the following hold.
\begin{enumerate}
\item\textnormal{(=\ref{B gives def locally 2})} The completion of $\sDefAlg$ at a closed point $z\in\nonIso$ is morita equivalent to the algebra that prorepresents noncommutative deformations of the reduced fibre over~$z$.
\item\textnormal{(=\ref{sheaf in deg 0})} $\twistOb$ is a sheaf.
\item\textnormal{(=\ref{supp equal excp})} $\Supp_X\twistOb$ equals the exceptional locus of $f$.
\item\textnormal{(=\ref{recovers universal locally})} The restriction of $\twistOb$ to the formal fibre above a closed point $z\in\nonIso$ recovers the universal sheaf $\twistOb_z$ from noncommutative deformation theory, up to additive equivalence.
\end{enumerate}
\end{thm}
The fact that in the one-dimensional fibre setting $\twistOb$ is a sheaf  is our main motivation for considering $\vTiltAlg^{\op}$ instead of $\vTiltAlg$.

\subsection{Applications to Autoequivalences $\twist$}\label{section construct twist}
We next turn our attention to crepant contractions and autoequivalences.  One benefit of the construction of $\twistOb$ on $X$ is that, using the natural exact sequence
\[
0\to\sIdeal\to\vTiltAlg^{\op}\to \sDefAlg\to 0,
\]
we can construct in \S\ref{global twists section},  under our most general setup in \S\ref{New inv section}, a functor $\twist$, which sits in a functorial triangle
\[
f^{-1} \RDerived{f}_* \RsHom_X(\twistOb,a) \Lotimes{f^{-1} \sDefAlg} \twistOb \to a \to\twist (a)\to.
\]
For a precise description of the terms in these expressions, we refer the reader to \S\ref{global twists section}, but remark here that when $\nonIso$ is a point, $\twist$ reduces to a noncommutative twist over a contraction algebra, as studied in \cite{DW1, DW3}.  The existence of a global sheafy contraction algebra $\sDefAlg$ allows us to avoid delicate local-to-global gluing arguments.

In general, $\twist$ is not an equivalence, but we have the following criterion in terms of a Cohen--Macaulay property of $\sDefAlg$ on $Y$, under a restriction to hypersurface singularities. There are two main reasons for this restriction on singularities (although it is not used everywhere), outlined in \ref{motivate hyper assump} below: note however that it holds automatically in the two main applications in this paper, namely to Springer resolutions of determinantal varieties of $n\times n$ matrices in \ref{det var thm}, and to $3$-fold divisor-to-curve contractions in \ref{one curve thm intro}.

\begin{thm}[=\ref{main sph thm}]\label{main twist 1 intro}
Under the general assumptions of \ref{key assumptions}, assume that $f$ is crepant, and $\widehat{\cO}_{Y,z}$ are hypersurfaces for all closed points $z\in\nonIso$. Then the following are equivalent.
\begin{enumerate}
\item\label{main twist 1 intro 1} $\sDefAlg$ is a Cohen--Macaulay sheaf on $Y$, and $\twistOb$ is a perfect complex on $X$.
\item\label{main twist 1 intro 2} $\sDefAlg$ is relatively spherical (in the sense of \ref{def geom sph}) for all closed points $z\in Z$.
\end{enumerate}
If these conditions hold, and they are automatic provided that $\dim Z\leq \dim Y-3$, then the functor $\twist$ is an autoequivalence of $\Db(\coh X)$.
\end{thm}

We remark that both parts of \ref{main twist 1 intro}\eqref{main twist 1 intro 1} can fail in general, and indeed if $\nonIso$ is not equidimensional, then $\sDefAlg$ is not relatively spherical for some $z\in Z$.

The last statement in \ref{main twist 1 intro} ensures that our framework recovers previous autoequivalences associated to flopping contractions~\cite{Toda flop, DW1, BB}, since the assumptions of \ref{main twist 1 intro} are satisfied in the one-dimensional fibre flops setting described in \cite[Theorem~C]{VdB1d}, but the main advantage of  \ref{main twist 1 intro} is that it includes other interesting settings, which we briefly  outline here.

Our first new application is to the varieties of singular $n\times n$ matrices, namely the determinantal varieties~$\mathds{k}[x_{ij}]/(\operatorname{det} x)$.  For each such variety, the Springer resolution admits a suitable tilting bundle by work of Buchweitz, Leuschke, and Van~den~Bergh~\cite{BLV}.

\begin{cor}[=\ref{det var thm main}]\label{det var thm}
Consider the Springer resolution $X\to Y$ of the variety of singular $n\times n$ matrices.  Then $\twist$ is an autoequivalence of $X$.
\end{cor}

We then establish results for one-dimensional fibre contractions which are not an isomorphism in codimension two, where the conditions in \ref{main twist 1 intro} are not automatic.  Amongst others, this includes partial resolutions of Kleinian singularities, and $3$-fold crepant divisor-to-curve examples. Leveraging our control of $\twistOb$ in this setting, which by \ref{main E results into} is a sheaf with support coinciding with the exceptional locus, we reformulate the first part of \ref{main twist 1 intro}\eqref{main twist 1 intro 1}, namely $\sDefAlg$ being a Cohen--Macaulay sheaf on $Y$, in terms of a Cohen--Macaulay property for $\twistOb$ on $X$.

\begin{thm}[=\ref{up and down CM}]\label{intro CM global}
With the general setup in \ref{main twist 1 intro}, suppose further that the fibres of $f$ have dimension at most one.  Then the following are equivalent.
\begin{enumerate}
\item $\sDefAlg$ is a Cohen--Macaulay sheaf on $Y$.
\item\label{intro CM global 2} For all $y\in \nonIso$, and all $x\in f^{-1}(y)$, we have $\twistOb_x\in\CM_{\dim \nonIso_y+1}\cO_{X,x}$.
\end{enumerate}
In particular, if these conditions hold, then above every  $y\in\nonIso$, the exceptional locus is equidimensional of dimension $\dim Z_y+1$.
\end{thm}

Combining with \ref{main twist 1 intro}, the dimension criterion above gives an easy-to-check obstruction to $\sDefAlg$ being relatively spherical, and we demonstrate this in various examples, such as~\ref{poky geom example}. Furthermore, the conditions above enable us to prove that $\twist$ is a equivalence in the motivating divisor-to-curve contraction example.

\begin{thm}[=\ref{one curve thm}]\label{one curve thm intro}
With the setup in \ref{intro CM global}, suppose in addition that $X$ is smooth and $\dim X=3$, such that every reduced fibre above a closed point in $Z$ contains precisely one irreducible curve.  Then 
$\twist$ is an autoequivalence of $X$. 
\end{thm}

\begin{remark}
We expect that, when $\twist$ is an autoequivalence, it can be expressed as a twist of a spherical functor as in \cite{Anno,Addington,AL2,Kuznetsov}. However, we do not address this question in this paper, as it would not simplify our proofs.
\end{remark}

\subsection{Acknowledgements} We are grateful for conversations with
Arend Bayer, 
Agnieszka Bodzenta, 
and Yukinobu Toda. 

\subsection{Conventions}  
Throughout we work over an algebraically closed field $\K$. Unqualified uses of the word `module' refer to right modules, and $\mod A$ denotes the category of finitely generated right $A$-modules.  We use the functional convention for composing homomorphisms, so $f\circ g$ means $g$ then $f$.  In particular, naturally this makes $M\in \mod A$ into an $\End_A(M)^{\op}$-module.

For $a \in \cA$ an abelian category, we let $\add a$ denote all possible summands of finite direct sums of $a$. Given two objects $a,b\in\cA$ where $a$ is a summand of $b$, we then write $[a]$ for the two-sided ideal of $\End_\cA(b)$ consisting of all morphisms that factor through $\add a$.

\section{Global Thickenings}\label{section global thickening}

In this section we will noncommutatively enhance the non-isomorphism locus of a contraction $f\colon X\to Y$ which satisfies some mild conditions.  We first recall what this means, and fix the setting.

\begin{defin}\label{setupglobal}
By a contraction, we will mean a projective birational morphism $f\colon X\to Y$ between normal varieties over $\K$, satisfying $\Rf_*\cO_X=\cO_{Y}$.  It will be assumed that $Y$ is quasi-projective.
\end{defin}

\begin{notation}
Given a contraction $f\colon X\to Y$, write $\nonIso$ for the locus of (not necessarily closed) points of $Y$ above which $f$ is not an isomorphism.  
\end{notation}

Given a contraction, we will furthermore assume the following condition.
\begin{assumption}\label{key assumptions}
For a contraction $f\colon X\to Y$, where $d = \dim X\geq 2$, assume that there is vector bundle $\vlocGen=\cO_X\oplus\vlocGen_0$ on $X$, such that the following conditions hold.
\begin{enumerate}
\item\label{key assumptions 1} The natural map $f_*\sEnd_X(\vlocGen)\to \sEnd_Y(f_*\vlocGen)$ is an isomorphism.
\item\label{key assumptions 2} The bundle $\vlocGen$ is tilting relative to $Y$, namely
\[
\RDerived{f}_* \sEnd_X(\vlocGen) = f_*\sEnd_X(\vlocGen)
=: \vTiltAlg,
\]
and furthermore there is an equivalence
\begin{equation}
\label{global Db}
\RDerived{f}_* \RsHom_X(\vlocGen,\placeholder) \colon \Db(\coh X) \overset{\sim}{\longrightarrow} \Db(\coh (Y, \vTiltAlg)).
\end{equation}
\end{enumerate}
Here $\vTiltAlg$ is considered as a sheaf of $\cO_Y$-algebras in the sense of~\cite[\S18.1]{KS}, and the bounded derived category of modules over $\vTiltAlg$ is denoted by $\Db(\coh (Y, \vTiltAlg))$.
\end{assumption}
With mind to our applications, we will show in \ref{assumptions do hold} below that the assumption is rather mild, and holds in general settings.  This requires the following easy consequence of \ref{key assumptions}\eqref{key assumptions 2}.
\begin{lemma}\label{easy tilt lemma}
Suppose that \ref{key assumptions}(\ref{key assumptions 2}) holds, $V$ is an affine open subset of $Y$, and $U:=f^{-1}(V)$.  Then $\vlocGen|_U$ is a tilting bundle on $U$.
\end{lemma}
\begin{proof}
Setting $\Lambda:=\End_{U}(\vlocGen|_U)$, the following diagram commutes.
\[
\begin{array}{c}
\begin{tikzpicture}
\node (roof) at (0,0) {$\Db(\coh X)$};
\node (X) at (5,0) {$\Db(\coh(Y,\vTiltAlg))$};
\node (X') at (0,-1.5) {$\Db(\coh U )$};
\node (base) at (5,-1.5) {$\Db(\mod \Lambda)$};
\draw[->] (roof) -- node[above] {$\scriptstyle \RDerived{f}_* \RsHom_X(\vlocGen,-)$} node[below] {$\scriptstyle \sim$} (X);
\draw[->] (roof) --  node[right] {$\scriptstyle |_U$}  (X');
\draw[->] (X) --  node[right] {$\scriptstyle |_V$} (base);
\draw[->] (X') -- node[above] {$\scriptstyle \RHom_U(\vlocGen|_U,-)$}   (base);
\end{tikzpicture}
\end{array}
\]
In particular, $\RHom_U(\vlocGen|_U,\vlocGen|_U)\cong\vTiltAlg|_V\cong\Lambda$, which gives Ext vanishing.  For generation, suppose that $a\in\D(\Qcoh U)$ with $\RHom_U(\vlocGen|_U,a)=0$.  We must show that $a=0$. By adjunction $\RHom_X(\vlocGen,\RDerived i_*a)=0$ where $i\colon U\hookrightarrow X$ is the open inclusion.  Since $\vlocGen$ compactly generates $\D(\Qcoh X)$, it follows that $\RDerived i_*a=0$, and thus $a=0$ since $\RDerived i_*$ is fully faithful. 
\end{proof}

\begin{prop}\label{assumptions do hold}
Assumption~\ref{key assumptions} is guaranteed to hold in the following two settings.
\begin{enumerate}
\item\label{assumptions do hold 1} $Y$ is a Gorenstein $d$-fold, $f$ is crepant, and $X$ admits a relative tilting bundle containing $\cO_X$ as  a summand.
\item\label{assumptions do hold 2} $X$ is a $d$-fold,  and the fibres of $f$ have dimension at most one.
\end{enumerate}
\end{prop}
\begin{proof}
(1) Take $\vlocGen$ to be the relative tilting bundle on $X$, with $\vlocGen=\cO_X \oplus \vlocGen_0$. Condition~\ref{key assumptions}(\ref{key assumptions 2})  is immediate. Condition~\ref{key assumptions}(\ref{key assumptions 1}) is local, therefore it suffices to consider affine $Y=\Spec R$, in which case the condition translates into showing that
\begin{equation}
\End_X(\vlocGen)\to\End_R(f_*\vlocGen),\label{up to down 1}
\end{equation}
is an isomorphism. Note that $\End_X(\vlocGen)$ is a maximal Cohen--Macaulay $R$-module since $f$~is crepant \cite[4.8]{IW5}.  We next claim that $f_*\vlocGen$ is also maximal Cohen--Macaulay.   The tilting property for $\vlocGen$ implies that $H^i(\vlocGen_0)=0=H^i(\vlocGen_0^*)$ for all $i>0$, and so
\begin{align*}
\RHom_R(f_*\vlocGen_0,R)&\cong \RHom_R(\Rf_*\vlocGen_0,\cO_R) \\
& \cong \Rf_*\RsHom_X(\vlocGen_0,f^!\cO_R) \\
& \cong \Rf_*\RsHom_X(\vlocGen_0,\cO_X) \\
& =\Rf_*(\vlocGen_0^*)
=f_*(\vlocGen_0^*),
\end{align*}
where we use Grothendieck duality and crepancy. Thus $\RHom_R(f_*\vlocGen_0,R)$ is concentrated in degree zero, hence $f_*\vlocGen_0$ is maximal Cohen--Macaulay, and thence $f_*\vlocGen$ is also.  In particular, \eqref{up to down 1} is a morphism between reflexive $R$-modules, so since it is an isomorphism in codimension two, it must be an isomorphism.\\
(2) Let $\vlocGen$ denote the bundle constructed by Van den Bergh in \cite[3.3.2]{VdB1d}, which is relatively generated by global sections.  Condition~\ref{key assumptions}(\ref{key assumptions 2}) is \cite[3.3.1]{VdB1d}.  For condition~\ref{key assumptions}(\ref{key assumptions 1}), note first that $f_*$ gives rise to the natural morphism 
\[
\vTiltAlg=f_*\sEnd_X(\vlocGen)\to\sEnd_Y(f_*\vlocGen).
\]
It suffices to check that this is an isomorphism on open affines $V$ of $Y$, on which it is just the following, where we write $U:=f^{-1}(V)$. 
\begin{equation}
f_*\colon \End_U(\vlocGen|_U)\to \End_{V}(f_*\vlocGen|_{V})\label{up=down eq}
\end{equation}
By \ref{easy tilt lemma} $\vlocGen|_U$ is tilting, so $\Ext^1_U(\vlocGen|_U,\vlocGen|_U)=0$.  Since further $Y$ is normal, $f$ is projective birational, and $\vlocGen|_U$ is generated by global sections, it follows from \cite[\citetype{Prop }4.3]{DW2}  that \eqref{up=down eq} is an isomorphism. 
\end{proof}

\begin{remark}
The relative tilting assumption from \ref{assumptions do hold}\eqref{assumptions do hold 1} reduces, in the case of affine $Y = \Spec R$, to the condition that $\Ext_X^i(\vlocGen,\vlocGen)=0$ for $i>0$, and $\vlocGen$ generates.
\end{remark}

\begin{remark}\label{Craw--Ishii}
We remark that if $d=3$ in \ref{assumptions do hold}\eqref{assumptions do hold 1}, the assumption that $X$ admits a tilting bundle is automatic if $X$ is constructed as the moduli of an NCCR of an affine $Y$~\cite[\S 6]{VdB2}.  It is expected that the tilting bundle condition in \ref{assumptions do hold}\eqref{assumptions do hold 1} for $d=3$ always holds, as a consequence of the Craw--Ishii conjecture and work of Iyama and the second author~\cite[4.18(2)]{IW4}.
\end{remark}

\subsection{Construction of $\sIdeal$}\label{subsection main constr} In this subsection we work under the general assumptions of \ref{key assumptions}, and define an ideal subsheaf $\sIdeal$ of $\vTiltAlg:=f_*\sEnd_X(\vlocGen)\cong\sEnd_Y(f_*\vlocGen)$. To ease notation, we put $\cQ:=f_*\vlocGen$, so that $\vTiltAlg\cong\sEnd_Y(\cQ)$. For $\cF$ a sheaf, $s\in\cF(V)$, and $v\in V$, we write $\langle s,V\rangle_v$ for $s$ viewed in the stalk $\cF_v$. 
\begin{defin}\label{defin sheaf of ideals}
With notation as above, for each open subset $V$\! of $Y$,  define
\[
\sIdeal(V):=\{ s\in\End_{V}(\cQ|_{V})\mid 
\cQ_v\xrightarrow{\langle s,V\rangle_v}\cQ_v \mbox{ factors through }\add\cO_{Y,v} \mbox{ for all }v\in V\},
\]
which is a two-sided ideal of $\vTiltAlg(V)$. That is, $\sIdeal$ consists of local sections of $\vTiltAlg$ that at each stalk factor through a finitely generated projective $\cO_{Y,v}$-module.
\end{defin}

\begin{prop}\label{is sheaf of ideals}
$\sIdeal$ is a subsheaf of $\vTiltAlg$ consisting of two-sided ideals.
\end{prop}
\begin{proof}
We just need to prove that $\sIdeal$ is a subsheaf.
For opens $U\subseteq V$ of $Y$, denote the restriction morphisms of $\vTiltAlg$ by $\rho^{\phantom U}_{VU}$. Note that if $s\in\sIdeal(V)$ then since
\[
\langle s,V\rangle_u=\langle \rho^{\phantom U}_{VU}(s),U\rangle_u,
\]
and this factors through a projective for all $u\in U$, it follows that $\rho^{\phantom U}_{VU}$ takes $\sIdeal(V)$ to $\sIdeal(U)$.  The presheaf axioms of $\vTiltAlg$ then imply that $\sIdeal$ is a subpresheaf of $\vTiltAlg$, so it suffices to check the two sheaf axioms.

Suppose then that $U=\bigcup U_i$ and $s\in\sIdeal(U)$ is such that $s|_{U_i}=0$ for all $i$.  Viewing this in $\vTiltAlg$, it follows that $s=0$.

Lastly, suppose that $U=\bigcup U_i$ and there is a collection $s_i\in\sIdeal(U_i)\subseteq\vTiltAlg(U_i)$ such that
\[
s_i|_{U_i\cap U_j}=s_j|_{U_i\cap U_j}
\]
for all $i,j$. Then since $\vTiltAlg$ is a sheaf certainly there exists $s\in\vTiltAlg(U)$ such that $s|_{U_i}=s_i$ for all $i$.  We claim that $s\in\sIdeal(U)$, that is  $\langle s,U\rangle_u$ factors through $\add\cO_{Y,u}$ for all $u\in U$. Thus let $u\in U$, then certainly $u\in U_i$ for some $i$.  But then
\[
\langle s,U\rangle_u=\langle s|_{U_i},U_i\rangle_u,
\]
which factors through a projective since $s|_{U_i}=s_i\in\sIdeal(U_i)$.
\end{proof}

We now show that on affine open subsets, $\sIdeal$ takes a particularly nice form.  To simplify notation, for an affine open subset $V=\Spec R$ of $Y$, write $Q$ for the $R$-module corresponding to $\cQ|_{\Spec R}$, so that $\vTiltAlg(\Spec R)\cong\End_R(Q)$.  Choose a surjection $h\colon F\twoheadrightarrow Q$ with $F$ a free $R$-module, and write
\[
\Hom_R(Q,F)\xrightarrow{\alpha:=(h\circ)}\Hom_R(Q,Q).
\]
It is standard that this localises to
\[
\Hom_{R_\p}\!(Q_\p,F_\p)\xrightarrow{\alpha_\p=(h_\p\circ)}\Hom_{R_\p}\!(Q_\p,Q_\p).
\]

By the defining property of projective modules, any morphism from an object in $P\in \add R$ to $Q$ has to factor as
\[
\begin{tikzpicture}
\node (A) at (0,0) {$F$};
\node (B2) at (-1,0) {$P$};
\node (B3) at (1,0) {$Q,$};
\draw[->>] (A) -- node[above] {\scriptsize $h$} (B3);
\draw[densely dotted,->] (B2) -- (A);
\end{tikzpicture}
\]
and any morphism from an object of $P\in \add R_\p$ to $Q_\p$ has to factor as
\[
\begin{tikzpicture}
\node (A) at (0,0) {$F_\p$};
\node (B2) at (-1,0) {$P$};
\node (B3) at (1.2,0) {$Q_\p.$};
\draw[->>] (A) -- node[above] {\scriptsize $h_\p$} (B3);
\draw[densely dotted,->] (B2) -- (A);
\end{tikzpicture}
\]
Hence
\begin{equation}
\{ g\colon Q\to Q\mid g\mbox{ factors through }\add R\}=\Im(\alpha),\label{approx loc 2}
\end{equation}
and
\begin{align}
\sIdeal(\Spec R)
&=\{g\colon Q\to Q\mid g_\p\mbox{ factors through }\add R_\p\mbox{ for all }\p\in\Spec R\}\nonumber\\
&=\{g\colon Q\to Q\mid g_\p\in \Im(\alpha_\p)\mbox{ for all }\p\in\Spec R\}\label{approx loc}.
\end{align}
The following lemma says that a morphism factors stalk-locally if and only if it factors affine-locally.
\begin{lemma}\label{affine locally it is ok}
If $\Spec R$ is an affine open subset of $Y$, then
\[
\sIdeal(\Spec R)=\{ g\colon Q\to Q\mid g\mbox{ factors through }\add R\}.
\]
\end{lemma}
\begin{proof}
By \eqref{approx loc 2} and \eqref{approx loc}, it suffices to prove that \[g\in \Im\alpha\Longleftrightarrow g_\p\in\Im\alpha_\p\]
for all $\p\in\Spec R$. The direction ($\Rightarrow$) is clear. For the ($\Leftarrow$) direction, suppose that $g\in\Im(\alpha_\p)$ for all $\p\in\Spec R$, and consider $g+\Im\alpha\in \End_R(Q)/\Im\alpha$.  Then  stalk-locally this element is zero, hence it is zero (see e.g.\ \cite[Lemma~2.8(a)]{EisenbudBook}), and thus $g\in\Im\alpha$. 
\end{proof}

It follows that we can define $\sIdeal$ in various equivalent ways, without referring to stalks.
\begin{cor}\label{3 ways to def}
Suppose that $V$ is an open subset of $Y$, and $s\in\vTiltAlg(V)$.  Then the following conditions are equivalent.
\begin{enumerate}
\item $s\in\sIdeal(V).$
\item There exists an open affine cover $V=\bigcup V_i$ such that $s|_{V_i}$ factors through $\add \cO_{V_i}$ for all $i$.
\item For every open affine cover $V=\bigcup V_i$,  $s|_{V_i}$ factors through $\add \cO_{V_i}$ for all $i$.
\end{enumerate}
\end{cor}
\begin{proof}
Since $\sIdeal$ is a sheaf by \ref{is sheaf of ideals}, if $s\in\sIdeal(V)$ then $s|_{V_i}\in\sIdeal(V_i)$.  Hence (1)$\Rightarrow$(3)  follows from \ref{affine locally it is ok}. (3)$\Rightarrow$(2) is obvious, and (2)$\Rightarrow$(1) is clear.
\end{proof}

\subsection{Sheafy Contraction Algebras}\label{sheafy contract alg section}
In this subsection we continue to work under the general assumptions of \ref{key assumptions}, but consider the dual bundle $\locGen:=\vlocGen^*$, and $\sTiltAlg:=f_*\sEnd_X(\locGen)$.  Although not strictly necessary (up to taking opposite algebras, we show in the proof of \ref{affine is nice} below that it gives the same objects), taking the dual is convenient since in the setting \ref{assumptions do hold}\eqref{assumptions do hold 2} of fibre dimension at most one, it will allow us to relate the above construction to our previous work, and deformation theory, in an easier way.

By \ref{is sheaf of ideals}, $\sIdeal$ is a sheaf of two-sided ideals of $\vTiltAlg:=f_*\sEnd_X(\vlocGen)$.  Since $\sTiltAlg\cong\vTiltAlg^{\op}$, we can also view $\sIdeal$ as a subsheaf of two-sided ideals of $\sTiltAlg$.  Thus we can consider the presheaf quotient of $\sTiltAlg$ by $\sIdeal$, which is naturally a presheaf of algebras, and hence its sheafification $\sDefAlg:=\sTiltAlg/\sIdeal$ is a sheaf of algebras on $Y$.

\begin{defin}\label{defin sheaf of deformation alg}
We call $\sDefAlg:=\sTiltAlg/\sIdeal$ the sheaf of contraction algebras on $Y$.
\end{defin}

Note that there is a natural $\cO_Y$-algebra structure on $\sDefAlg$  given by a morphism of sheaves of rings $\cO_Y \to \sTiltAlg$, as in \cite[\S18.1]{KS}, however this morphism is not injective, as a consequence of \ref{global cont thm} below.  Locally, the sections of $\sDefAlg$ have a particularly easy form, which we now describe.  For an affine neighbourhood $V$ in $Y$, consider the following Zariski local setup obtained by base change.

\begin{setup}\label{setupZariski}
(Zariski local) Suppose that $f\colon U\to V=\Spec R$  is a projective birational morphism between normal varieties over $\K$, with $\Rf_*\cO_U=\cO_V$.
\end{setup}

We set $\Lambda_V:=\End_U(\locGen|_U)$.  The following is an extension of \cite[\S 3.1]{DW2}, which only considered contractions with fibre dimension at most one.  

\begin{defin}\label{contraction algebra def}
Write $I_{\con}$ for the two-sided ideal of $\Lambda_V$ consisting of those morphisms $\varphi\colon \locGen|_U\to\locGen|_U$ which factor through an object $\cF\in\add \cO_U$, as follows.
\[
\begin{tikzpicture}
\node (A) at (0,0) {$\locGen|_U$};
\node (B2) at (1.5,-0.75) {$\cF$};
\node (C) at (3,0) {$\locGen|_U$};
\draw[densely dotted,->] (A) -- (B2);
\draw[densely dotted,->] (B2) -- (C);
\draw[->] (A) -- node[above] {\scriptsize $\varphi$} (C);
\end{tikzpicture}
\]
Then the \emph{contraction algebra} for~$\Lambda_V$ is defined to be $(\Lambda_V)_{\con}:=\Lambda_V/I_{\con}$.
\end{defin}

The following proposition is important, and will be heavily used later.
\begin{prop}\label{affine is nice}
If $V=\Spec R$ is an affine open subset of $Y$, then $\sDefAlg(V)\cong(\Lambda_{V})_{\con}$. 
\end{prop}
\begin{proof}
For all affine open subsets $V=\Spec R$ of $Y$
\[
0\to\sIdeal(V)\to\sTiltAlg(V)\to \sDefAlg(V)\to 0
\]
is exact. It follows that, for $U:=f^{-1}(V)$, 
\begin{align*}
\sDefAlg(V)
=\sTiltAlg(V)/\sIdeal(V)
&=\big(\vTiltAlg(V)/\sIdeal(V)\big)^{\op}\\
&\cong
\big(\!\End_R(f_*\vlocGen|_U)/[R]\big)^{\op}\tag{by \eqref{up=down eq}, \ref{affine locally it is ok}}\\
&=\big(\!\End_U(\vlocGen|_U)/[\cO_U]\big)^{\op}\tag{by \ref{key assumptions}\eqref{key assumptions 1}}
\end{align*}
which is $\End_U(\locGen|_U)/[\cO_U]=(\Lambda_V)_{\con}$.
\end{proof}

A benefit of our global construction of $\sDefAlg$, which we will make use of later, is that the contraction theorem in \cite{DW2} can be globalised.  Although \cite[4.4--4.7]{DW2} was stated in the one-dimensional fibre setting, the proof works word-for-word in the more general setup here.

\begin{cor}[Global Contraction Theorem]\label{global cont thm}
$\Supp_Y\sDefAlg=\nonIso$.
\end{cor}
\begin{proof}
By \cite[4.7]{DW2}, which does not require the one-dimensional fibre assumption, we see that $\Supp_V(\Lambda_V)_{\con}=\nonIso\cap V$ for all affine opens $V$ in $Y$.   Since $\Supp_V(\Lambda_V)_{\con}=\Supp_V\sDefAlg|_V$ by \ref{affine is nice}, the result follows.
\end{proof}

\begin{remark}\label{nc enhancement remark}
By construction $\sDefAlg$ is a sheaf of algebras on $Y$.  By \ref{global cont thm} it is naturally a sheaf of algebras on $\nonIso$, and thus we can view the ringed space $(\nonIso,\sDefAlg)$ as a noncommutative enhancement of the locus $\nonIso$.  Note further that $\coh(Y,\sDefAlg)=\coh(\nonIso,\sDefAlg)$.
\end{remark}

\begin{example}\label{ex D4 to A2}
Consider $R:=\mathbb{C}[x,y,z,t]/(x^2t + y^3 - z^2)$ and $Y=\Spec R$. This is singular along $x=y=z=0$, and has a crepant resolution sketched below, obtained by blowing up the ideal $(x,y,z)$.
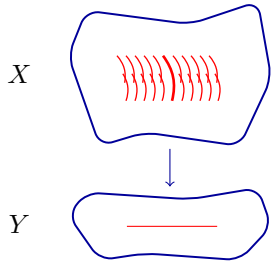
\[
\begin{array}{c}
\begin{tikzpicture}
\node at (0,0) {\begin{tikzpicture}[scale=0.6]
\coordinate (T) at (1.9,2);
\coordinate (TM) at (2.12-0.05,1.5-0.1);
\coordinate (BM) at (2.12-0.09,1.5+0.1);
\coordinate (B) at (2.1,1);
\draw[red,line width=1pt] (T) to [bend left=25] (B);
\foreach \y in {0.2,0.4,...,1}{ 
\draw[line width=0.5pt,red] ($(T)+(\y,0)+(0.02,0)$) to [bend left=25] ($(TM)+(\y,0)+(0.02,0)$);
\draw[line width=0.5pt,red] ($(BM)+(\y,0)+(0.02,0)$) to [bend left=25] ($(B)+(\y,0)+(0.02,0)$);
\draw[line width=0.5pt,red] ($(T)+(-\y,0)+(-0.02,0)$) to [bend left=25] ($(TM)+(-\y,0)+(-0.02,0)$);
\draw[line width=0.5pt,red] ($(BM)+(-\y,0)+(-0.02,0)$) to [bend left=25] ($(B)+(-\y,0)+(-0.02,0)$);;}
\draw[color=blue!60!black,rounded corners=5pt,line width=0.75pt] (0.5,0) -- (1.5,0.3)-- (3.6,0) -- (4.3,1.5)-- (4,3.2) -- (2.5,2.7) -- (0.2,3) -- (-0.2,2)-- cycle;
\end{tikzpicture}};
\node at (0,-2) {\begin{tikzpicture}[scale=0.6]
\draw [red] (1.1,0.75) -- (3.1,0.75);
\draw[color=blue!60!black,rounded corners=5pt,line width=0.75pt] (0.5,0) -- (1.5,0.15)-- (3.6,0) -- (4.3,0.75)-- (4,1.6) -- (2.5,1.35) -- (0.2,1.5) -- (-0.2,1)-- cycle;
\end{tikzpicture}};
\draw[->, color=blue!60!black] (0,-1) -- (0,-1.5);
\node at (-2,0) {$X$};
\node at (-2,-2) {$Y$};
\end{tikzpicture}
\end{array}
\]
In this example $X$ is derived equivalent to $\Lambda:=\End_R(R\oplus M)$, where $M$ is the cokernel of the following $4\times 4$ matrix.
\begin{equation}
R^4
\xrightarrow{
\left(
\begin{array}{cccc}
x &z&-y&0\\
z &xt&0&y\\
y^2&0&xt&z\\
0&y^2&-z&-x
\end{array}
\right)}
R^4\label{MF in ex}
\end{equation}

We now calculate the stalks of the sheaf $\sDefAlg$ at the closed points of $\nonIso$.   It is clear that, away from the origin, complete locally $R$ is given by the equation $x^2 + y^3 - z^2$. This is the $A_2$ surface singularity crossed with the affine line, and thus the generic fibre is just the minimal resolution of $A_2$ crossed with the affine line.  Since $M$ has rank two, it follows that for all points of $\nonIso$ away from the origin, the completion of the stalk of $\sDefAlg$ is given by the completion of the following quiver with relations.
\[
\begin{array}{cccc}
\begin{array}{c}
\begin{tikzpicture} [bend angle=45, looseness=1]
\node (C1) at (0,0) [vertex] {};
\node (C2) at (1.5,0) [vertex] {};
\node (C1a) at (-0.1,0)  {};
\node (C2a) at (1.6,0) {};
\draw [->,bend left=20,looseness=1] ($(C1)+(20:5pt)$) to node[gap]  {$\scriptstyle a$} ($(C2)+(160:5pt)$);
\draw [->,bend left=20] ($(C2)+(-160:5pt)$) to node[gap]  {$\scriptstyle b$} ($(C1)+(-20:5pt)$);
\draw[->]  (C1a) edge [in=150,out=220,loop,looseness=10,pos=0.5] node[left] {$\scriptstyle s$} (C1a);
\draw[->]  (C2a) edge [in=30,out=-40,loop,looseness=10,pos=0.5] node[right] {$\scriptstyle t$} (C2a);
\end{tikzpicture}
\end{array}
&&
{\small \begin{array}{l}
a\circ s=t\circ a\\
b\circ t=s\circ b\\
a\circ b=0\\
b\circ a=0
\end{array}}
\end{array}
\]

On the other hand, the origin is a $cD_4$ point, and calculating the stalk of $\sDefAlg$ there is a little trickier.  Using the matrix in \eqref{MF in ex}, it is easy to first present $\Lambda$ as a quiver with relations, as in \cite[\S 4]{AM}, then calculate the contraction algebra, as in \cite[1.3]{DW1}.  Doing this, one finds that  the completion of $\sDefAlg$ at the origin is the completion of 
\[
\frac{\mathbb{C}\langle a,b\rangle}{ab+ba,a^2}
\]
at the ideal $(a,b)$.
\end{example}

By \ref{nc enhancement remark} there is a noncommutative ringed space $(\nonIso,\sDefAlg)$, and it is natural to compare this to the usual  commutative ringed space $(\nonIso,\cO_{\nonIso})$, where $\cO_{\nonIso}$ is the structure sheaf that endows $\nonIso$ with its reduced subscheme structure.  A key property of $(\nonIso,\cO_{\nonIso})$ is that for every point $z\in\nonIso$, the stalk $\cO_{\nonIso,z}$ is local, that is, it admits only one simple module.  This is not true in general for the stalk $\sDefAlg_z$.  However, in the crepant setting of \ref{assumptions do hold}\eqref{assumptions do hold 1}, the following holds. 

\begin{prop}\label{loc D and hyper}
Suppose the crepant setting of \ref{assumptions do hold}(\ref{assumptions do hold 1}), let $z\in\nonIso$ be a closed point, and furthermore assume that $d=3$ and $X$ is smooth.  If $\widehat{\sDefAlg}_{z}$ is local, then $\widehat{\cO}_{Y,z}$ is a hypersurface.
\end{prop}
\begin{proof}
Put $\mathfrak{R}:=\widehat{\cO}_{Y,z}$, and write $\widehat{\sDefAlg}_z=\End_{\mathfrak{R}}(\mathfrak{R}\oplus M)$.  Since by assumption $\widehat{\sDefAlg}_{z}$ is local, necessarily $M$ must be indecomposable.  Then by \cite[6.25(3)]{IW4} mutation at the indecomposable summand $M$ must be an involution, and so $\Omega^2M\cong M$.  But this implies that the complexity of $M$ must be less than or equal to one, which easily (see for example the proof of \cite[6.14]{HomMMP}) implies that $\mathfrak{R}$ is a hypersurface.
\end{proof}

\begin{remark}\label{motivate hyper assump}
The converse to \ref{loc D and hyper} is false, since if $R$ is complete locally a hypersurface, $\widehat{\sDefAlg}_{z}$ need not be local.   Nevertheless, being the only situation with smooth $X$ where a local $\widehat{\sDefAlg}_{z}$ is possible, the situation where $Y$ is complete locally a hypersurface has special status.  It is also the only situation where our $\sDefAlg$ can possibly coincide with Toda's noncommutative thickening \cite{Toda thickening}. This partially motivates the hypersurface assumptions in \S\ref{Section brutal} below.
\end{remark}

\subsection{$\sDefAlg$ and Deformations}\label{subsection deformations}
In the setting \ref{assumptions do hold}\eqref{assumptions do hold 2} when $f$ has fibres of dimension at most one, in this subsection we briefly relate $\sDefAlg$ to noncommutative deformation theory.  For this, recall that an \emph{$n$-pointed $\K$-algebra} $\Gamma$ is an associative $\K$-algebra equipped with $\K$-algebra morphisms
\[
\K^n \overset{i}{\to} \Gamma \overset{p}{\to} \K^n
\]
which compose to the identity. Note that the morphisms $p$ and $i$ allow us to lift the canonical idempotents $\{e_1, \dots, e_n\}$ of $\K^n$ to $\Gamma$. We refer to $\Gamma$ as \emph{artinian} if it is finite-dimensional as a vector space over $\K$, and the ideal $\Ker p$ is nilpotent. Such algebras naturally form a category $\art_n$, and we write $\proart_n$ for the category of pro-artinian algebras (see for instance \cite[\S 2]{DW2} for the precise construction). 

Recall that a DG category is a graded category $\mathsf{A}$ whose morphism spaces have the structure of DG vector spaces. We write $\DG_n$ for the category which has as objects the DG categories with precisely $n$ objects. Given a DG category $\mathsf{A} \in \DG_n$ and an algebra $\Gamma \in \art_n$, we now recall an appropriate notion of deformations over $\Gamma$, according to the standard Maurer--Cartan formulation. Writing $\n = \Ker p \subset \Gamma$, first consider the DG category $\mathsf{A} \uotimes \n \in \DG_n$ with objects $\{ 1, \dots, n \}$, morphisms 
\[
\Hom_{\mathsf{A} \uotimes \n}(i,j) := \Hom_{\mathsf{A}}(i,j) \otimes_{\K} (e_j \n e_i),
\]
and differential induced from $\mathsf{A}$. (Note that the convention used here for notating compositions in $\Gamma$ is opposite to that of \cite{DW2}.) Given a degree-$1$ morphism $\xi \in (\mathsf{A} \uotimes \n)^1$ consider the Maurer--Cartan equation 
\[
\operatorname{MC}(\xi) := d\xi + \tfrac{1}{2}[\xi,\xi] \in (\mathsf{A} \uotimes \n)^2 = 0
\]
where $d$ is the differential, and the bracket is induced from the commutator brackets on $\mathsf{A}$ and $\n$.

\begin{defin} Given $\mathsf{A} \in \DG_n$, the associated \emph{deformation functor} is given by
\begin{align*}
\Def^{\mathsf{A}} \colon \art_n & \to \Sets \\
\Gamma & \mapsto \left\{ \xi \in (\mathsf{A} \uotimes \n)^1 \,|\, \operatorname{MC}(\xi) = 0 \right\} / \sim
\end{align*}
where the standard gauge equivalence relation $\sim$ is given in e.g.\ \cite[2.6]{DW2}.
\end{defin}

Given objects $E_1, \dots, E_n \in \coh(X)$, and injective resolutions $E_i \to I^i_\bullet$, then 
\[
\mathsf{A} := \End^{\DG}\left(\bigoplus I^i_\bullet\right)
\]
naturally lies in $\DG_n$, and the associated deformation functor $\Def^{\mathsf{A}}$ describes the simultaneous noncommutative deformations of the collection $\{E_i\}$.

\begin{defin} A functor $F\colon \art_n \to \Sets$ is said to be \emph{prorepresentable} by $\Gamma \in \proart_n$ if the restriction of $\Hom_{\proart_n}\!(\Gamma, \placeholder)$ to $\art_n$ is naturally isomorphic to $F$.
\end{defin}

 Given a contraction $f$ with at most one-dimensional fibres as in \ref{assumptions do hold}\eqref{assumptions do hold 2}, for every $z\in\nonIso$,  $C=f^{-1}(z)$ is a curve, and recall that we write $C^{\redu} = \bigcup_{i=1}^n C_i$ where each $C_i \cong \mathbb{P}^1$. We put $E_i = \cO_{C_i}(-1)$ and form $\mathsf{A}$ as above.  

The following theorem was shown in \cite[1.1]{DW2}.

\begin{thm}
In the setting \ref{assumptions do hold}(\ref{assumptions do hold 2}), for each closed point $z \in \nonIso$ the functor of noncommutative deformations $\Def^{\mathsf{A}}$ is represented by an algebra $\CAz$.
\end{thm}

The precise form of $\CAz$ is not important for now, but we do explain this in \S\ref{setting in Section 3}.  Recall from \S\ref{sheafy contract alg section} that for an affine open $V$ of $Y$, and $U:=f^{-1}(V)$, we write $\Lambda_V:=\End_U(\vlocGen^*|_U)$.  It was shown in \cite[\citetype{Theorem~}1.1, \citetype{Theorem~}1.2]{DW2} that the completion of $\sDefAlg|_V=(\Lambda_V)_{\con}$ at a closed point $z\in \nonIso\cap V$ is morita equivalent to $\CAz$.  The following is then an immediate corollary of \ref{affine is nice}.

\begin{cor}\label{B gives def locally 2} 
In the setting \ref{assumptions do hold}(\ref{assumptions do hold 2}), let $z\in Z$ be a closed point. Then the completion of the stalk $\sDefAlg_z$ is morita equivalent to $\CAz$. 
\end{cor}

The morita equivalence above is illustrated in the following example.

\begin{example}\label{D4 to A1 ex}
Consider $R:=\mathbb{C}[x,y,z,t]/(x^3-xyt-y^3+z^2)$ and $Y=\Spec R$, which is singular along the $t$-axis.  It has a crepant resolution sketched as follows.
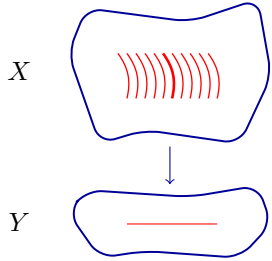
\[
\begin{array}{c}
\begin{tikzpicture}
\node at (0,0) {\begin{tikzpicture}[scale=0.6]
\coordinate (T) at (1.9,2);
\coordinate (B) at (2.1,1);
\draw[red,line width=1pt] (T) to [bend left=25] (B);
\foreach \y in {0.2,0.4,...,1}{ 
\draw[line width=0.5pt,red] ($(T)+(\y,0)$) to [bend left=25] ($(B)+(\y,0)$);
\draw[line width=0.5pt,red] ($(T)+(-\y,0)$) to [bend left=25] ($(B)+(-\y,0)$);;}
\draw[color=blue!60!black,rounded corners=6pt,line width=0.75pt] (0.5,0) -- (1.5,0.3)-- (3.6,0) -- (4.3,1.5)-- (4,3.2) -- (2.5,2.7) -- (0.2,3) -- (-0.2,2)-- cycle;
\end{tikzpicture}};
\node at (0,-2) {\begin{tikzpicture}[scale=0.6]
\draw [red] (1.1,0.75) -- (3.1,0.75);
\draw[color=blue!60!black,rounded corners=6pt,line width=0.75pt] (0.5,0) -- (1.5,0.15)-- (3.6,0) -- (4.3,0.75)-- (4,1.6) -- (2.5,1.35) -- (0.2,1.5) -- (-0.2,1)-- cycle;
\end{tikzpicture}};
\draw[->, color=blue!60!black] (0,-1) -- (0,-1.5);
\node at (-2,0) {$X$};
\node at (-2,-2) {$Y$};
\end{tikzpicture}
\end{array}
\]
In this example $X$ is derived equivalent to $\Lambda=\End_R(R\oplus M)$, where $M$ is the cokernel of the following $4\times 4$ matrix.
\[
R^4
\xrightarrow{
\left(
\begin{array}{cccc}
x & z& -y& 0\\
0&y^2&-z&x\\
-z&x^2-yt&0&y\\
-y^2&0&x^2-yt&z
\end{array}
\right)}
R^4
\]

Away from the origin, the singular locus is just the $A_1$ surface singularity crossed with the affine line, so since $M$ has rank two, complete locally away from the origin it must split into two isomorphic copies of the same rank-one CM module $L$, so that
\[
\widehat{\Lambda}\cong\End_{\mathfrak{R}}(\mathfrak{R}\oplus L^{\oplus 2}).
\]
Hence, away from the origin, the completion of the stalk of $\sDefAlg$ is the $2\times 2$ matrices over $\mathbb{C}[t]$, since $\mathbb{C}[t]$ is the contraction algebra of $\End_{\mathfrak{R}}(\mathfrak{R}\oplus L)$.

At the origin, complete locally $M$ is indecomposable, so that the stalk of $\sDefAlg$ at the origin must be a local algebra.  Using the same method as in \ref{ex D4 to A2}, it is not difficult to show that the completion of the stalk is isomorphic to the completion of $\mathbb{C}\langle a,b\rangle/(a^2,b^2)$ at the ideal $(a,b)$.
\end{example}

\begin{remark}\label{not Toda}
This remark explains why our thickening is different to the one constructed by Toda \cite{Toda thickening}.  In the above example, a generic exceptional fibre consists of a single smooth projective curve. Take its structure sheaf and consider its Hilbert polynomial to obtain a moduli space of simple sheaves.  All generic fibres have this Hilbert polynomial (by flatness), and the moduli space has a point for each one of these, plus a point corresponding to the origin, to give a pure dimension one scheme.  The completion of Toda's sheaf $\sDefAlg'$ at each point abelianizes to give the completion of the stalk of the structure sheaf of the moduli space at that point.  In particular, since the moduli space has pure dimension one, this stalk cannot be finite dimensional.  But $\sDefAlg$ abelianizes at the central point to give $\mathbb{C}[[a,b]]/(a^2,b^2)$, which is finite dimensional, and hence the sheaves $\sDefAlg$ and $\sDefAlg'$ are different.  
\end{remark}

\section{A Universal Sheaf $\twistOb$}
In this section we use $\sDefAlg$ from \ref{defin sheaf of deformation alg} to construct a complex of sheaves $\twistOb$ on $X$, and present some of its basic properties.  In the case of fibre dimension at most one, \ref{assumptions do hold}\eqref{assumptions do hold 2}, we will show that $\twistOb$ is a sheaf, with support equal to the exceptional locus.

\subsection{General Construction of $\twistOb$}\label{general cE construction section}
In this subsection we work under the general assumptions of \ref{key assumptions}.  By definition, \ref{defin sheaf of deformation alg}, $\sDefAlg$ is a sheaf on $Y$.  However, in the setting \ref{assumptions do hold}\eqref{assumptions do hold 2}, consider $\locGen:=\vlocGen^*$, which induces an equivalence
\begin{equation}
\begin{array}{c}
\begin{tikzpicture}
\node (d1) at (-3,0) {$\Db(\coh X)$};
\node (e1) at (3,0) {$\Db(\coh (Y, \sTiltAlg)).$};
\draw[->] (d1) to  node[above] {$\scriptstyle \RDerived{f}_* \RsHom_X(\locGen,\placeholder) $} node[below] {$\scriptstyle \sim$}(e1);
\end{tikzpicture}
\end{array}\label{global Db us}
\end{equation}
The sheaf of algebras $\sDefAlg:=\sTiltAlg/\sIdeal$ constructed in \ref{defin sheaf of deformation alg} is in particular an $\sTiltAlg$-module, and thus an object of $\Db(\coh (Y,\sTiltAlg))$.  
\begin{notation}\label{cE def}
We write $\twistOb$ for the object in $\Db(\coh X)$ which corresponds to $\sDefAlg$ across the equivalence \eqref{global Db us}.
\end{notation}

\begin{prop}\label{derived chase}
Under the general assumptions of \ref{key assumptions}, for an affine open $V=\Spec R\subseteq Y$, set $U=f^{-1}(V)$ and write $f'\colon U\to V$ for the restricted map.  Then the following hold.
\begin{enumerate}
\item\label{derived chase 1} $\RDerived f'_*(\twistOb|_U)=0$.
\item\label{derived chase 2}  $\Rf_*\twistOb=0$.
\end{enumerate}
In particular, $\Supp_X\twistOb$ is contained in the exceptional locus.
\end{prop}
\begin{proof} 
(1) There is a commutative diagram as follows.
\begin{equation}
\begin{array}{c}
\begin{tikzpicture}
\node (roof) at (0,0) {$\Db(\coh X)$};
\node (X) at (5,0) {$\Db(\coh(Y,\sTiltAlg))$};
\node (X') at (0,-1.5) {$\Db(\coh U )$};
\node (base) at (5,-1.5) {$\Db(\mod \Lambda_V)$};
\draw[->] (roof) -- node[above] {$\scriptstyle \RDerived{f}_* \RsHom_X(\locGen,-)$} node[below] {$\scriptstyle \sim$} (X);
\draw[->] (roof) --  node[right] {$\scriptstyle |_U$}  (X');
\draw[->] (X) --  node[right] {$\scriptstyle |_V$} (base);
\draw[->] (X') -- node[above] {$\scriptstyle \RHom_U(\locGenZar,-)$} node[below] {$\scriptstyle \sim$}  (base);
\end{tikzpicture}
\end{array}\label{global Db us local}
\end{equation}
Since $\twistOb$ corresponds to $\sDefAlg$ on the top equivalence, it follows that $\twistOb|_U$ corresponds to $\sDefAlg(V)$ across the bottom equivalence.  Hence by \ref{affine is nice} $\twistOb|_U$ corresponds to $(\Lambda_V)_{\con}$.
 
Let $e$ denote the idempotent in $\Lambda_V$ corresponding to $R$. Then the following diagram commutes. 
\begin{equation}
\begin{array}{c}
\begin{tikzpicture}
\node (roof) at (0,0) {$\Db(\coh U)$};
\node (X) at (5,0) {$\Db(\mod\Lambda_V)$};
\node (X') at (0,-1.5) {$\Db(\coh \Spec R)$};
\node (base) at (5,-1.5) {$\Db(\mod R)$};
\draw[->] (roof) -- node[above] {$\scriptstyle \RHom_U(\locGenZar,-)$} node[below] {$\scriptstyle \sim$} (X);
\draw[->] (roof) --  node[right] {$\scriptstyle \RDerived f'_*$}  (X');
\draw[->] (X) --  node[right] {$\scriptstyle (-)e$} (base);
\draw[-,transform canvas={yshift=+0.15ex}] (X') -- (base);
\draw[-,transform canvas={yshift=-0.15ex}] (X') -- (base);
\end{tikzpicture}
\end{array}\label{local comm push}
\end{equation}
Since $(\Lambda_V)_{\con}\in\mod\Lambda_V$ satisfies $(\Lambda_V)_{\con}e=0$, it follows that $\twistOb|_U$ satisfies $\RDerived f'_*(\twistOb|_U)=0$.  \\
(2) Since  $\mathbf{R}^i f_*\twistOb|_V=\mathbf{R}^i f'_*(\twistOb|_{U})$ by flat base change (see e.g.\ \cite[III.8.2]{Hartshorne}), it follows from \eqref{derived chase 1} that $\Rf_*\twistOb=0$. 
\end{proof}

It may be the case that $\Supp_X\twistOb$ always equals the exceptional locus, but this seems tricky to prove without more control over where the simple $(\Lambda_V)_{\con}$-modules go under the derived equivalence.   Controlling the support of $\twistOb$ is important, since later it will give an easy-to-check obstruction to $\sDefAlg$ being relatively spherical.

\subsection{The Sheaf $\twistOb$ and Fibre Dimension One}\label{sect univ sheaf} 
This subsection considers the setting of \ref{assumptions do hold}\eqref{assumptions do hold 2} when $f$ has fibres of dimension at most one, and proves that $\twistOb$ is a sheaf whose support $\Supp_X\twistOb$ equals the exceptional locus.

In the setting of \ref{assumptions do hold}\eqref{assumptions do hold 2}, recall from~\cite[\S 3]{Bridgeland} and~\cite[\S 3.1]{VdB1d} that perverse coherent sheaves on $X$ may be defined as follows.

\begin{defin}
The category $\PerOne(X,Y)$, respectively $\Per(X,Y)$, consists of those objects  $a\in\Db(\coh X)$ such that 
\begin{enumerate}
\item $H^i(a)=0$ if\, $i\neq 0,-1$,
\item $f_*H^{-1}(a)=0$, $\Rfi{1}_* H^0(a)=0$,
\item $\Hom_X(H^0(a),c)=0$, respectively $\Hom_X(c,H^{-1}(a))=0$, for all $c\in \cC^0$
\end{enumerate}
where $\cC:=\ker \Rf_*$ and $\cC^0$ denotes the full subcategory of $\cC$ whose objects have cohomology only in degree $0$.
\end{defin}

It follows from \cite[\citetype{Proposition~}3.3.1, proof of \citetype{Proposition~}3.3.2]{VdB1d} that the $\vlocGen$ in \ref{assumptions do hold}\eqref{assumptions do hold 2} is a local progenerator of $\PerOne(X/Y)$, and $\locGen=\vlocGen^*$ likewise for $\Per(X/Y)$.

To show that  $\twistOb$ is a sheaf requires the following lemma, which  is well known.

\begin{lemma}\label{Per lemma track}
Suppose that $a\in\Db(\coh X)$ such that $\Rf_* a=0$.
\begin{enumerate}
\item\label{Per lemma track 1} $a\in\PerOne(X,Y)$ if and only if $a$ is concentrated in degree $-1$.
\item\label{Per lemma track 2} $a\in\Per(X,Y)$ if and only if $a$ is concentrated in degree $0$.
\end{enumerate}
\end{lemma}
\begin{proof}
(1) ($\Rightarrow$) By definition $H^i(a)=0$ unless $i=-1\text{ or }0$, so we need only show $H^0(a)=0$. As $f$ has at most one-dimensional fibres, $\Rf_* a=0$ implies that $\Rf_* H^0(a)=0$ using \cite[\citetype{Lemma~}3.1]{Bridgeland}. But  $\Hom_X(H^0(a),c)=0$ for all sheaves $c$ such that $\Rf_* c=0$, in particular for $c=H^0(a)$, and the claim follows.\\
($\Leftarrow$) This follows immediately from the definition.\\
(2) The proof is similar.
\end{proof}

\begin{cor}\label{sheaf in deg 0}
In the setting of \ref{assumptions do hold}(\ref{assumptions do hold 2}), the following statements hold.
\begin{enumerate}
\item $\twistOb|_U\cong
(\Lambda_V)_{\con}\otimes^{}_{\Lambda_V}\locGenZar$.
\item $\twistOb$ is a sheaf in degree zero.
\end{enumerate}
\end{cor}
\begin{proof}
(1) This follows from the proof of~\ref{derived chase}, in particular~\eqref{global Db us local}.\\
(2) This is shown using \ref{derived chase} and \ref{Per lemma track}\eqref{Per lemma track 2} applied to $\Per(U,V)$.
\end{proof}

For a closed point $z\in \nonIso$, consider an affine neighbourhood $V=\Spec R$ of $z$, and consider $U:=f^{-1}(V)$ and $f|_U\colon U\to V$.  The fibre above each closed point in~$\nonIso$, with reduced scheme structure, decomposes into curves $C_i$ say, each isomorphic to $\mathbb{P}^1$.

\begin{prop}\label{supp closed pts}
In the setting of \ref{assumptions do hold}(\ref{assumptions do hold 2}), if $x$ is a closed point in the exceptional locus of~$f|_U$, then $x\in \Supp_U \twistOb|_U$.  
\end{prop}
\begin{proof}
Certainly $x$ sits on some $C_i$, so $x\in\Supp \cO_{C_i}(-1)$. Write $T$ for the $\Lambda_V$-module corresponding to $\cO_{C_i}(-1)$ across the bottom equivalence in \eqref{global Db us local}, then by \cite[3.5.8]{VdB1d} $T$~is a simple $\Lambda_V$-module.  But since $\Rf_*\cO_{C_i}(-1)=0$, it follows from \eqref{local comm push} that $T$ is also a $(\Lambda_V)_{\con}$-module, and hence a simple $(\Lambda_V)_{\con}$-module.  As such, there exists a surjection $(\Lambda_V)_{\con}\twoheadrightarrow T$, and so back across the bottom equivalence in \eqref{global Db us local} there is a short exact sequence
\[
0\to \cK\to \twistOb|_U\to\cO_{C_i}(-1)\to 0
\]
in $\Per(U,V)$. Since the last two terms are sheaves, it follows from the long exact sequence in ordinary cohomology that so too is $\cK$.  Hence since passing to stalks is exact, the stalk of $ \twistOb$ at $x$ must be non-zero, since it surjects onto the stalk of $\cO_{C_i}(-1)$ at $x$, which is non-zero.  It follows that $x\in\Supp\twistOb|_U$. 
\end{proof}

\begin{cor}\label{supp equal excp}
In the setting of \ref{assumptions do hold}(\ref{assumptions do hold 2}), $\Supp_X \twistOb$ is equal to the exceptional locus.
\end{cor}
\begin{proof}
By \ref{supp closed pts}, by varying $z\in \nonIso$ we see that the set of closed points in the support of $\twistOb$ contains the set of closed points in the exceptional locus. Since $\Rf_*\twistOb=0$ by \ref{derived chase}\eqref{derived chase 2}, $\Supp\twistOb$ does not contain any further closed points. Since $X$ is a variety, and so Zariski locally every prime ideal is the intersection of maximal ideals, we are done. 
\end{proof}

We next show that the global $\twistOb$ recovers the universal sheaves from noncommutative deformation theory in the setting of \ref{assumptions do hold}\eqref{assumptions do hold 2}.  For a closed point $z\in\nonIso$, write $\mathfrak{R}_z$ for the completion of the structure sheaf of $Y$ at $z$, and consider the following flat base change diagram.
\[
\begin{array}{c}
\begin{tikzpicture}[yscale=1.25]
\node (Xp) at (-1,0) {$\mathfrak{X}_z$}; 
\node (X) at (1,0) {$X$};
\node (Rp) at (-1,-1) {$\Spec\mathfrak{R}_z$}; 
\node (R) at (1,-1) {$Y$};
\draw[->] (Xp) to node[above] {$\scriptstyle j_z$} (X);
\draw[->] (Rp) to node[above] {$\scriptstyle i_z$} (R);
\draw[->] (Xp) --  node[left] {$\scriptstyle f_z$}  (Rp);
\draw[->] (X) --  node[right] {$\scriptstyle f$}  (R);
\end{tikzpicture}
\end{array}
\]
The deformation theory in \S\ref{subsection deformations} gives rise to a universal sheaf $\twistOb_z\in\coh\mathfrak{X}_z$ for every $z\in\nonIso$.

\begin{prop}\label{recovers universal locally}
In the setting of \ref{assumptions do hold}(\ref{assumptions do hold 2}), for every $z\in\nonIso$, $\add j_z^*\twistOb=\add \twistOb_z$.
\end{prop}
\begin{proof}
This is implicit in \cite[3.7]{DW2}, but we sketch the argument here. To ease notation, write $\Lambda=\Lambda_V$.   

As in \cite[3.7]{DW2} the algebra $\widehat{\Lambda}$ is morita equivalent to an algebra $\AB$, via a functor $\mathbb{F}$, and furthermore there are decompositions as $\widehat{\Lambda}$-modules
\begin{equation}
\mathbb{F}\CA=P_1\oplus\hdots\oplus P_n
\quad\mbox{and}\quad
\widehat{\Lambda}_{\con}=P_1^{\oplus a_1}\oplus\hdots\oplus P_n^{\oplus a_n}\label{complete decomps}
\end{equation}
for some $a_i\geq 1$, where the $P_i$ are the projective $\widehat{\Lambda}_{\con}$-modules. Now
\[
\RHom_{\mathfrak{X}_z}(j_z^*\locGen,\placeholder)\colon 
\Db(\coh \mathfrak{X}_z)\to\Db(\mod\widehat{\Lambda})
\]
is an equivalence, and it is easy to see using flat base change that $j_z^*\twistOb$ corresponds to $\widehat{\Lambda}_{\con}$.  As in \cite[4.14]{HomMMP} and \cite[\S3.2]{DW1}, $\twistOb_z$ corresponds across the equivalence to $\mathbb{F}\CA$.  Since \eqref{complete decomps} implies that $\add\widehat{\Lambda}_{\con}=\add\mathbb{F}\CA$, it follows that $\add j_z^*\twistOb=\add \twistOb_z$.
\end{proof}

\section{Spherical Properties via Cohen--Macaulay Modules}\label{Section brutal}

This section considers the crepant contractions of \ref{assumptions do hold}\eqref{assumptions do hold 1} for the Zariski local case and characterises, under some assumptions on singularities, when the noncommutative enhancement $\sDefAlg$ is relatively spherical.  These results are globalised in \S\ref{global twists section}.  All this involves the Cohen--Macaulay property, which we now review.

\subsection{Cohen--Macaulay Modules}\label{section CM}
Recall that if $(R,\m)$ is a commutative local noetherian ring, with $M\in\mod R$, then the \emph{depth} of $M$ is defined to be
\[
\depth_R M:=\mathrm{min}\{ i\mid \Ext^i_R(R/\m,M)\neq 0\},
\]
and for all $\p\in\Ass(M)$ there is a chain of inequalities
\[
\depth_RM\leq \dim R/\p \leq \dim_RM.
\]
Then $M$ is called a \emph{Cohen--Macaulay} module if either $M=0$, or $M\neq 0$ and 
\[
\depth_R M=\dim_RM.
\]
We write $\CM_\cS R$ for the category of such modules.  It is clear that if $M$ is Cohen--Macaulay, necessarily $M$ is equidimensional (i.e.\ $\dim R/\p=\dim_R M$ for every minimal prime $\p\in\Supp M$), and has no embedded primes (i.e.\ every associated prime is minimal).

If $R$ is local Gorenstein, and $M\neq 0$, then by local duality (see e.g.\ \cite[3.5.11]{BH})
\begin{align*}
M\in\CM_{\cS} R
&\iff \RHom_R(M,R)\mbox{ is concentrated in degree }\dim R-\dim_R M.
\end{align*}
When $R$ is not necessarily local, but is still Gorenstein and equi-codimensional, we say that $M\in\mod R$ is CM if $M_\m\in\CM_\cS R_\m$ for all maximal ideals $\m$ of $R$, and again write $\CM_{\cS}R$ for the category of CM $R$-modules.  Since $\dim_{R_\m}\! M_\m$ can vary between maximal ideals,  $\RHom_R(M,R)$ may have cohomology in more than one degree.

We say that $M\in\mod R$ is \emph{maximal Cohen--Macaulay} if $\depth_{R_\m}\! M_\m=\dim R_\m$ for all $\m\in\Max R$, and we write $\MCM R$ for the category of maximal Cohen--Macaulay $R$-modules.

\subsection{Setting}\label{setting in Section 3}
The remainder of this section considers the following refinement of the crepant contraction setting of \ref{assumptions do hold}\eqref{assumptions do hold 1} to the affine case.

\begin{setup}\label{general modifying setup}
Suppose that $f\colon X\to Y$ is a contraction as in \ref{setupglobal}, where in addition $Y=\Spec R$ is affine and Gorenstein, $d = \dim X \geq 2$, $f$ is crepant, and $X$ admits a tilting bundle $\vlocGen=\cO_X\oplus \vlocGen_0$. 
\end{setup}

Necessarily $R$ is equi-codimensional by \cite[\citetype{Corollary }13.4]{EisenbudBook}, and a Gorenstein normal domain by assumption.
\begin{notation} 
We set $M:=f_*\vlocGen_0$, which is an $R$-module, and consider $\Lambda:=\End_R(R\oplus M)$, which is isomorphic to $\End_X(\vlocGen)$ by \ref{assumptions do hold}\eqref{assumptions do hold 1}.  Further, since $f$ is crepant necessarily $\Lambda\in\MCM R$ by \cite[4.8]{IW5}, and thus $M\cong\Hom_R(R,M)\in\MCM R$, being a summand of~$\Lambda$.  We write $\Lambda_{\con}:=\Lambda/I_{\con}$, where $I_{\con}:=[R]$.
\end{notation}

We will often reduce arguments about $\Lambda_{\con}$ to the completions of closed points, in which case the following notation will be useful.  For a closed point $z=\m\in\Spec R$, write $\mathfrak{R}$ for the completion of $R$ at $z$, and consider the Krull--Schmidt decomposition
\[
\widehat{M}=\mathfrak{R}^{\oplus a_0}\oplus M_1^{\oplus a_1}\oplus\hdots\oplus M_n^{\oplus a_n}.
\]
Then we write $K=M_1\oplus\hdots\oplus M_n$, and $\AB:=\End_{\mathfrak{R}}(\mathfrak{R}\oplus K)$.  This depends on $z$, but we suppress this from the notation.  We define
\[
\CAz:=\AB/[\mathfrak{R}],
\]
but throughout this section, to ease notation we will usually refer to this as simply $\CA$.

\begin{defin}\label{def sph}
We say that $\Lambda_{\con}$ is \emph{$t$-relatively spherical} if
\[
\Ext^j_{\Lambda}(\Lambda_{\con},T)\cong\left\{  
\begin{array}{cl}
\K&\mbox{if }j=0,t\\
0&\mbox{else},
\end{array}
\right.
\]
for all simple $\Lambda_{\con}$-modules $T$.  There is no requirement that $\Lambda_{\con}$ is perfect.  
\end{defin}

There is an obvious variant of \ref{def sph} for $\CA$.  The following two subsections characterise when $\Lambda_{\con}$ and $\CA$ are relatively spherical, under the assumption that complete locally $R$ has only hypersurface singularities.  This additional assumption is motivated in part by~\ref{motivate hyper assump}, in part by the fact that in characteristic zero $3$-dimensional Gorenstein terminal singularities have this property \cite[0.6(I)]{Pagoda}, and in part since one of our main applications later in \S\ref{div to curve section} will be to crepant divisor-to-curve contractions of $3$-folds, in which case the hypersurface singularity condition holds automatically.

\subsection{Spherical via CM $R$-modules I} 
The following three results are elementary.

\begin{lemma}\label{support lemma}
In the setup \ref{general modifying setup}, with $\p\in\Spec R$, the following are equivalent.
\begin{enumerate}
\item $\p\in \nonIso$.
\item $\p\in\Supp\Lambda_{\con}$.
\item $M_\p$ is a non-free $R_\p$-module.
\end{enumerate}
\end{lemma}
\begin{proof}
(1)$\Leftrightarrow$(2) is the local version of \ref{global cont thm}.\\
(2)$\Leftrightarrow$(3) $(\Lambda_{\con})_\p\cong\uEnd_{R_{\p}}(M_\p)$, which is zero if and only if $M_\p$ is free.
\end{proof}

The following is also elementary, and is a simple consequence of the depth lemma.
\begin{lemma}\label{Ext vanishes}
In the setup \ref{general modifying setup},  if $\dim_R\Lambda_{\con}\leq d-3$, then $\Ext^1_R(M,M)=0$.
\end{lemma}
\begin{proof}
By the assumption and \ref{support lemma}, for $\p\in\Spec R$ with $\hgt\p=2$, $M_\p$ is a free $R_\p$-module, and hence certainly $\Ext^1_{R_\p}(M_\p,M_\p)=0$.  This implies that for all $\q\in\Spec R$ with $\hgt\q=3$, $\Ext^1_{R_\q}(M_\q,M_\q)$ is a finite length $R_\q$-module.

But since $\End_{R_\q}(M_\q)\in\MCM R_\q$, by the depth lemma $\Ext^1_{R_\q}(M_\q,M_\q)=0$.  In turn, this implies that the Ext group for height four primes has finite length.  Again, by the depth lemma, this must be zero.  By induction the result follows.
\end{proof}

\begin{cor}\label{cor sing}
In the setup \ref{general modifying setup},  if  $\dim_R\Sing R\leq d-3$, then $\Ext^1_R(M,M)=0$.
\end{cor}
\begin{proof}
Since $M\in\CM R$, it is clear from \ref{support lemma} that $\dim_R\Lambda_{\con}\leq \dim_R\Sing R$.  The result then follows immediately from  \ref{Ext vanishes}.
\end{proof}

The remainder of this subsection considers the case when $\Lambda_{\con}$ does not have maximal dimension, that is when $\dim_R\Lambda_{\con}\leq d-3$.  The case  $\dim_R\Lambda_{\con}= d-2$ is trickier, and will be dealt with in the next subsection.

\begin{lemma}\label{case 3}
In the setup \ref{general modifying setup}, suppose further that $Y$ is complete locally a hypersurface.  If $\dim_R\Lambda_{\con}\leq d-3$, then at every closed point $z\in\nonIso$,
\begin{enumerate}
\item\label{case 3 1} $\pd_\AB\CAz=3$.
\item\label{case 3 2}  $\CAz\in\CM_{d-3}\mathfrak{R}$, in particular $\dim_{\mathfrak{R}}\CAz=d-3$.
\setcounter{tempenum}{\theenumi}
\end{enumerate}
Furthermore, the following statements hold.
\begin{enumerate}
\setcounter{enumi}{\thetempenum}
\item\label{case 3 3} The stalk of $\cO_Y$ at every closed point of $Z$ has dimension $d-3$.
\item\label{case 3 4} $\pd_\Lambda\Lambda_{\con}=3$.
\item\label{case 3 5} $\Lambda_{\con}$ is $3$-relatively spherical.
\end{enumerate}
\end{lemma}
\begin{proof}
(1)(2) By matrix factorisation there is an exact sequence
\[
0\to K\to F\to F\to K\to 0,
\]
where $F$ is a free $\mathfrak{R}$-module. By assumption and \ref{Ext vanishes}, we know that $\Ext^1_{\mathfrak{R}}(K,K)=0$, hence applying $\Hom_{\mathfrak{R}}(\mathfrak{R}\oplus K,-)$ to the above sequence gives a projective resolution
\[
0\to P_K\to P_0^{\oplus b}\to P_0^{\oplus b}\to P_K\to\CA\to 0.
\]
Thus $\pd_\AB\CA\leq 3$, from which Auslander--Buchsbaum implies that $\depth\CA\geq d-3$.  It thus follows that $\CA\in\CM_{d-3}\mathfrak{R}$, and so $\pd_\AB\CA=3$.\\
(3) By \eqref{case 3 2} $\dim_{\mathfrak{R}}\CA=d-3$, and by \ref{support lemma} $\dim_{\mathfrak{R}}\CA=\dim_{\mathfrak{R}}\nonIso$. \\
(4) This follows from \eqref{case 3 1}, since projective dimension can be checked complete locally.\\
(5) This follows by checking complete locally, and using the projective resolution of $\CA$ given above.
\end{proof}

\subsection{Spherical via CM $R$-modules II}  Seeking a more general version of \ref{case 3} when $\Lambda_{\con}$ has maximal dimension is subtle, for two reasons.  First, $\pd_\Lambda\Lambda_{\con}=\infty$ can occur, and second the dimension of $Z$ at closed points may vary, in which case asking for $\Lambda_{\con}\in\CM_{\cS} R$ is more natural than specifying a particular $\CM_{d-t}R$.

Before extracting a global statement, we first work complete locally, and extend \ref{case 3} as follows.

\begin{thm}\label{spher prop complete}
In the setup \ref{general modifying setup}, suppose further that $R$ is complete locally a hypersurface.  Then $\CA$ is $t$-relatively spherical if and only if
\begin{enumerate}
\item\label{spher prop complete 1} $\pd_\AB\CA<\infty$.
\item\label{spher prop complete 2} $\CA\in\CM_{d-t}\mathfrak{R}$.
\end{enumerate}
In this case necessarily $t=d-\dim_{\mathfrak{R}}\CA$, which is either $2$ or $3$, and furthermore the assumptions \eqref{spher prop complete 1} and \eqref{spher prop complete 2} hold when $\dim_{\mathfrak{R}}\CA\leq d-3$.
\end{thm}
\begin{proof}
($\Leftarrow$) \emph{Case} $t=2$:  By Auslander--Buchsbaum, $\pd_\AB\CA=2$.  It then follows from \cite[A.3]{HomMMP} that $\Omega K\cong K$, and so applying $\Hom_{\mathfrak{R}}(\mathfrak{R}\oplus K,-)$ to the exact sequence
\[
0\to \Omega K\to F\to K\to 0
\]
gives the minimal projective resolution
\[
0\to P_K\to P_0^{\oplus a}\to P_K\to\CA\to 0.
\]
Hence $\CA$ is $2$-relatively spherical, and since by assumption $\CA\in\CM_{d-2}\mathfrak{R}$, we have~$t=d-\dim\CA$.\\
\emph{Case} $t\geq 3$:  Since $\CA\in\CM_t\mathfrak{R}$, necessarily $\dim_{\mathfrak{R}}\CA\leq d-3$, so by \ref{case 3} $\CA$ is \mbox{$3$-relatively} spherical, and $t=d-\dim_{\mathfrak{R}}\CA$.\\
($\Rightarrow$) \emph{Case} $t=2$: Consider the beginning of the minimal projective resolution of $\CA$.  Since $\CA$ is $2$-relatively spherical, this has the form
\[
\to P_K\to P_0^{\oplus a}\xrightarrow{\psi} P_K\to\CA\to 0. 
\]
Certainly $\Ker\psi=\Hom_{\mathfrak{R}}(\mathfrak{R}\oplus K,\Omega K)$, hence the projective cover of $\Ker\psi$ is obtained by applying $\Hom_{\mathfrak{R}}(\mathfrak{R}\oplus K,-)$ to a minimal $\add(\mathfrak{R}\oplus K)$-approximation
\[
f\colon U\to \Omega K.
\]
But since the projective cover of $\Ker\psi$ is $P_K$, it follows that $\Hom_{\mathfrak{R}}(\mathfrak{R}\oplus K,U)\cong P_K:=\Hom_{\mathfrak{R}}(\mathfrak{R}\oplus K,K)$, so by reflexive equivalence $U\cong K$.  Further, since $f$ is in particular an $\add \mathfrak{R}$-approximation, it is necessarily surjective. But $\mathfrak{R}$ is a hypersurface, so the rank of $K$ equals the rank of $\Omega K$, hence $\Ker f=0$ and thus $f$ is an isomorphism.  In turn, this implies that $\Ker\psi\cong P_K$, and hence $\pd_\AB\CA=2$.  

By Auslander--Buchsbaum, $\depth_{\mathfrak{R}}\CA=d-2$.  But as in \ref{cor sing}, and since $\mathfrak{R}$ is normal 
\[
d-2=\depth_{\mathfrak{R}}\CA\leq \dim_{\mathfrak{R}}\CA\leq \dim\Sing \mathfrak{R}\leq d-2.
\]
Hence equality holds throughout, and  $\CA\in\CM_{d-2}\mathfrak{R}$.\\
\emph{Case} $t\geq 3$: Again, consider the beginning of the minimal projective resolution of $\CA$, which  now has the form
\[
\to P_0^{\oplus b}\xrightarrow{\phi} P_0^{\oplus c}\xrightarrow{\psi} P_K\to\CA\to 0.
\]
The morphism $P_0^{\oplus b}\to\Ker\psi=\Hom_{\mathfrak{R}}(\mathfrak{R}\oplus K,\Omega K)$ is induced from a morphism of the form $f\colon F\to \Omega K$ (where $F$ is free), 
which is a minimal $\add({\mathfrak{R}}\oplus K)$-approximation.  Again, necessarily this $f$ is surjective, and its kernel is $\Omega\Omega K\cong K$, since $\mathfrak{R}$ is a hypersurface.  But then, applying $\Hom_{\mathfrak{R}}(\mathfrak{R}\oplus K,-)$ to the exact sequence
\[
0\to K\to F\xrightarrow{f} \Omega K\to 0
\]
and using the fact that $f$ is an $\add (\mathfrak{R}\oplus K)$-approximation, 
\[
0\to \Hom_{\mathfrak{R}}({\mathfrak{R}}\oplus K,K)\to \Hom_{\mathfrak{R}}({\mathfrak{R}}\oplus K,F)\twoheadrightarrow \Hom_{\mathfrak{R}}({\mathfrak{R}}\oplus K,\Omega K)\to \Ext^1_{\mathfrak{R}}(K,K)\to 0
\]
is exact.  It follows that $\Ext^1_{\mathfrak{R}}(K,K)=0$, and $\Ker\phi=P_K$.  Consequently, $\pd_\AB\CA=3$ and $\CA$ is $3$-relatively spherical, thus $t=3$.

By Auslander--Buchsbaum, $\depth_{\mathfrak{R}}\CA=d-3$.  We claim that $\CA\in \CM_{d-3}{\mathfrak{R}}$, so we just need to prove that $\dim_{\mathfrak{R}} \CA\neq d-2$.  If there exists $\p\in\Supp\CA$ with height two, then by \ref{support lemma} $K_\p$ is a non-free $R_\p$-module.  But by the above 
\[
\Hom_{\uCM R_\p}\!(K_\p,K_\p[1])\cong\Ext^1_{{\mathfrak{R}}_\p}\!(K_\p,K_\p)=0,
\]
and further since ${\mathfrak{R}}_\p$ is a Gorenstein surface with only an isolated singular point,  $\uCM {\mathfrak{R}}_\p$ is $1$-CY.  This implies that $\uHom_{{\mathfrak{R}}_\p}\!(K_\p,K_\p)=0$, and so $K_\p$ is free, which is a contradiction.  Hence such a $\p$ with height two cannot exist, and so $\CA\in\CM_{d-3}{\mathfrak{R}}$.

The last statement is \ref{case 3}.
\end{proof}

\begin{remark}
Neither condition \eqref{spher prop complete 1} or \eqref{spher prop complete 2} in \ref{spher prop complete}  is guaranteed if $\dim_{\mathfrak{R}}\CA=d-2$; see \ref{poky out alg example} below and also \cite[4.18]{HomMMP}.
\end{remark}

Continuing to work complete locally, we obtain the following tilting consequence.

\begin{prop}\label{complete tilting}
In the setup \ref{general modifying setup}, suppose further that $R$ is complete locally a hypersurface.  If the equivalent conditions of \ref{spher prop complete} hold, then $I_\AB$ is a tilting $\AB$-module.
\end{prop}
\begin{proof}
By \ref{spher prop complete} there are only two cases, $t=2$ and $t=3$.  When $t=2$, the fact that $I_\AB$ is tilting is just \cite[A.3, A.5]{HomMMP}.  When $t=3$, the argument is identical to \cite[5.10(1)]{DW1}, since although in that proof $d=3$, the fact that by \ref{Ext vanishes} $\Ext^1_R(K,K)=0$  means that \cite[(5.F)]{DW1} is still exact, so exactly the same proof works.
\end{proof}

It is the global version of \ref{complete tilting} that interests us the most. The fact that $\CA$ can be relatively spherical at all closed points, but that the value of the spherical parameter can vary, is problematic; similarly the projective dimension of $I_\AB$ can vary over the maximal ideals.  Hence we do not seek an autoequivalence condition globally in terms of one parameter. Instead, in condition \eqref{thm zariski local tilting 3} below we ask for $\Lambda_{\con}$ to belong to $\CM_{\cS} R$, which gives the required flexibility.  The following is one of our main results.
\begin{thm}\label{thm zariski local tilting}
In the setup \ref{general modifying setup}, suppose further that $R$ is complete locally a hypersurface.  Then the following are equivalent.
\begin{enumerate}
\item\label{thm zariski local tilting 1} $\CAz$ is a relatively spherical $\AB$-module for all $z\in\Max R$.
\item\label{thm zariski local tilting 2} $\pd_{\AB}\CAz<\infty$ and $\CAz\in\CM_{\cS} \mathfrak{R}$ for all $z\in\Max R$.
\item\label{thm zariski local tilting 3} $\pd_{\Lambda}\Lambda_{\con}<\infty$ and $\Lambda_{\con}\in\CM_{\cS} R$.
\end{enumerate}
If any of these conditions hold, and they are automatic if $\dim\Sing R\leq d-3$, then $I_{\con}$~is a tilting $\Lambda$-module, and $\placeholder\Lotimes{\Lambda} I_{\con}$ preserves $\Db(\mod\Lambda)$.
\end{thm}
\begin{proof}
(1)$\Leftrightarrow$(2) is \ref{spher prop complete}, since $d-t=d-(d-\dim_{\mathfrak{R}}\CA)=\dim_{\mathfrak{R}}\CA$.\\  (2)$\Leftrightarrow$(3) just follows since finite projective dimension can be checked complete locally at maximal ideals, and $\Lambda_{\con}\in\CM_{\cS} R$ is again defined locally. 

The statement about the assumptions being automatic is \ref{case 3}.  The fact $I_{\con}$ is a tilting module then follows since being a tilting module can be checked complete locally (see e.g.\ the proof of \cite[6.2]{DW1}), and the complete local statement is \ref{complete tilting}.
\end{proof}

\begin{example}\label{poky out alg example}
Consider $M:=(u,x)\oplus (u,x^2)$ for $R=\mathbb{C}[u,v,x,y]/(uv-x^2y)$.  Then $\Lambda:=\End_R(R\oplus M)$ is an NCCR of $R$.  Complete locally at the origin, since $\Omega M\ncong M$, as in \ref{spher prop complete} it follows that $\CA\notin\CM_{\cS}\mathfrak{R}$.  In fact, this can be seen directly, since the minimal projective resolution of $\CA$ has the form
\[
0\to P_2^{\oplus 2}\to P_0^{\oplus 3}\oplus P_1\to P_0^{\oplus 4}\to P_1\oplus P_2\to\CA\to 0
\]
so, by inspection, $\CA$ is not relatively spherical.   Note that since $\CA\notin\CM_{\cS}\mathfrak{R}$, it follows that $\Lambda_{\con}\notin\CM_{\cS}R$.  However, the other condition in \ref{thm zariski local tilting}\eqref{thm zariski local tilting 3}, namely $\pd_{\Lambda}\Lambda_{\con}<\infty$, does hold since $\Lambda$ is an NCCR.
\end{example}
A more conceptual geometric explanation of the above example is given in \ref{poky geom example} below.

\section{Global Twist Functors}\label{global twists section}

In this section, under our most general $d$-fold contraction assumptions of \ref{key assumptions}, where additionally $f$~is crepant, we produce an endofunctor $\twist$ of $\Db(\coh X)$.  After restricting singularities, we then give a criterion for $\twist$ to be an autoequivalence, before giving our first application.  Further applications will appear in \S\ref{1d fibre last section}.

\subsection{Twist Construction}

By the definition of $\sDefAlg$ in \ref{defin sheaf of deformation alg}, there is an exact sequence of $\sTiltAlg$-bimodules
\begin{equation}\label{eqn bimodule ses}
0\to\sIdeal\to\sTiltAlg\to\sDefAlg\to 0.
\end{equation}
The bimodule $\sIdeal$ induces a functor
\[
\RsHom_{\sTiltAlg}(\sIdeal,\placeholder)\colon
\D(\Mod\sTiltAlg)\to\D(\Mod\sEnd_{\kern -1pt\sTiltAlg} \,\sIdeal),
\] 
and the following lemma, which is simply the global version of \cite[6.1]{DW1}, ensures that it furthermore yields an endofunctor of $\D(\Mod\sTiltAlg)$.

\begin{lemma}\label{I both sides ok}
Under the general assumptions of \ref{key assumptions}, suppose further that $f$ is crepant.
\begin{enumerate}
\item\label{I both sides ok 1} There is an isomorphism of sheaves of algebras $\sTiltAlg \cong \sEnd_{\kern -1pt\sTiltAlg} \,\sIdeal$.
\item\label{I both sides ok 2} Under this  isomorphism the $(\sEnd_{\kern -1pt\sTiltAlg} \,\sIdeal,  \sTiltAlg)$-bimodule structure on $\sIdeal$ coincides with the natural $\sTiltAlg$-bimodule structure.  
\end{enumerate}
\end{lemma}
\begin{proof}{}
(1) There is a canonical morphism 
\begin{equation}\label{equation end tilting sheaf}
\sTiltAlg \rightarrow \sEnd_{\kern -1pt\sTiltAlg} \,\sIdeal
\end{equation}
given on any open subset $V$ of $Y$ by
\begin{align*}
\sTiltAlg(V) & \rightarrow \Hom_{\sTiltAlg(V)} (\sIdeal(V), \sIdeal(V)) \\
\lambda & \mapsto \alpha_\lambda
\end{align*}
with $\alpha_\lambda \colon i \mapsto \lambda i.$
By our assumptions \ref{key assumptions}, $\dim X\geq 2$, so since $f$ is crepant the proof of \cite[6.1(1)]{DW1} shows that this is an isomorphism for affine $V$, where \ref{affine locally it is ok} ensures that $\sIdeal(V)$ is the ideal considered in~\cite{DW1}. It follows that \eqref{equation end tilting sheaf} is an isomorphism.\\
(2) This is a formal consequence of \eqref{I both sides ok 1}, as in \cite[6.1(2)]{DW1}.
\end{proof}

Composing the above endofunctor with the equivalences
\begin{equation}
\begin{array}{c}
\begin{tikzpicture}
\node (e1) at (4,0) {$\D(\Mod\sTiltAlg)$};
\node (f1) at (10,0) {$\D(\Qcoh X)$};
\draw[<-,transform canvas={yshift=-0.5ex}] (e1) to node[below]   {$\scriptstyle G^{\RA}=\Rf_*\RsHom_{X}(\locGen,\placeholder)  $} (f1);
\draw[->,transform canvas={yshift=+0.5ex}] (e1) to  node[above] {$\scriptstyle G=f^{-1}(\placeholder) \Lotimes{f^{-1}\sTiltAlg} \locGen$} (f1);
\end{tikzpicture}
\end{array}\label{for intro 2}
\end{equation}
leads to the following definition.

\begin{defin}\label{defin geom twist} 
Under the general assumptions of \ref{key assumptions}, suppose further that $f$ is crepant.  The \emph{twist} and \emph{dual twist} endofunctors are defined to be
\[
\twist, \twist^* \colon \D(\Qcoh X) \to \D(\Qcoh X)
\]
\begin{align*}
\twist &= G \funcompose \RsHom_{\sTiltAlg}(\sIdeal,\placeholder)  \funcompose G^{\RA},\\
\twist^* &= G \funcompose (\placeholder \Lotimes{\sTiltAlg} \sIdeal) \funcompose G^{\RA}.
\end{align*}
\end{defin}
We will show that $\twist$ preserves $\Db(\coh X)$ in \ref{prop twist preserves Db}(\ref{prop twist preserves Db 2}) below, under the condition that $\sDefAlg$ is a perfect $\sTiltAlg$-module, or equivalently $\twistOb$ from  \ref{cE def} is an object in $\Perf(X)$.  To do this, we first describe some additional structure on $\twistOb$.

By definition $\twistOb = f^{-1} \sDefAlg \Lotimes{f^{-1} \sTiltAlg} \locGen$ is a complex of sheaves on $X$.  Computing this expression by resolving the second factor we see that $\twistOb\in\Db(\mod f^{-1}\sDefAlg^{\kern 0.1em \op})$.

Next, consider the following adjoint functors
\begin{equation}
\begin{array}{c}
\begin{tikzpicture}
\node (d1) at (0,0) {$\D(\Mod \sDefAlg)$};
\node (e1) at (4,0) {$\D(\Mod\sTiltAlg)$};
\node (f1) at (10,0) {$\D(\Qcoh X).$};
\draw[<-,transform canvas={yshift=-0.5ex}] (d1) to node [below]  {$\scriptstyle \RsHom_{\sTiltAlg}(\sDefAlg,\placeholder)  $} (e1);
\draw[->,transform canvas={yshift=+0.5ex}] (d1) to  node [above] {$\scriptstyle \placeholder \Lotimes{\sDefAlg} \sDefAlg$} (e1);
\draw[<-,transform canvas={yshift=-0.5ex}] (e1) to node [below]  {$\scriptstyle G^{\RA}=\Rf_*\RsHom_{X}(\locGen,\placeholder)  $} (f1);
\draw[->,transform canvas={yshift=+0.5ex}] (e1) to  node [above] {$\scriptstyle G=f^{-1}(\placeholder) \Lotimes{f^{-1}\sTiltAlg} \locGen$} (f1);
\end{tikzpicture}
\end{array}\label{for intro}
\end{equation}
By construction and assumption, $\sDefAlg$ and $\locGen$ are sheaves.  We will write $F$ for the composition of the top functors, and $F^{\RA}$ for the composition of the bottom functors. Note that $F$ and $F^{\RA}$ can be expressed easily as
\begin{align*}
F&=f^{-1}(\placeholder\Lotimes{\sDefAlg}\sDefAlg)\Lotimes{f^{-1}\sTiltAlg}\locGen \\
&\cong
f^{-1}(\placeholder)\Lotimes{f^{-1}\sDefAlg}f^{-1}\sDefAlg\Lotimes{f^{-1}\sTiltAlg}\locGen\tag{since $f^{-1}$ distributes over tensor} \\
&\cong
f^{-1}(\placeholder)\Lotimes{f^{-1}\sDefAlg}\twistOb,\\
&\\[-8pt]
F^{\RA}&=\RsHom_{\sTiltAlg}(\sDefAlg,\Rf_*\RsHom_{X}(\locGen,\placeholder))\\
&\cong\Rf_*\RsHom_{f^{-1}\sTiltAlg}(f^{-1}\sDefAlg,\RsHom_{X}(\locGen,\placeholder))\tag{by adjunction, c.f.\ \cite[(18.3.2)]{KS}}\\
&\cong\Rf_*\RsHom_{X}(f^{-1}\sDefAlg\Lotimes{f^{-1}\sTiltAlg}\locGen,\placeholder)\tag{by adjunction}\\
&=\Rf_*\RsHom_{X}(\twistOb,\placeholder).
\end{align*}

\begin{remark}
Regardless of whether $\twistOb$ is a sheaf or a complex, below and throughout when we write $\Rf_*\RsHom_{X}(\twistOb,\placeholder)$ we will mean the functor $F^{\RA}$, which takes values in $\D(\Mod\sDefAlg)$. In particular, 
$f^{-1}\Rf_*\RsHom_{X}(\twistOb,\placeholder)$ 
takes values in $\D(\Mod f^{-1}\sDefAlg)$.  
\end{remark}

\begin{prop}\label{funct triangle prop}
Under the general assumptions of \ref{key assumptions}, suppose further that $f$ is crepant.  Then $\twist$ fits into a functorial triangle
\[
f^{-1} \RDerived{f}_* \RsHom_X(\twistOb,\placeholder) \Lotimes{f^{-1} \sDefAlg} \twistOb \to \Id \to \twist \to.
\]
\end{prop}
\begin{proof}
For any object $a \in \D(\Mod\sTiltAlg)$,  simply applying $\RsHom_{\sTiltAlg}(-,a)$ to the sequence \eqref{eqn bimodule ses} gives a functorial triangle in $\D(\Mod\sTiltAlg)$
\begin{equation}\label{to be conj}
\RsHom_{\sTiltAlg}(\sDefAlg,a) \to a \to \RsHom_{\sTiltAlg}(\sIdeal,a) \to.
\end{equation}
We may reinterpret the left-hand term as 
$
\RsHom_{\sTiltAlg}(\sDefAlg,a) \Lotimes{\sDefAlg} \sDefAlg_{\sTiltAlg},
$
in other words, it is given by a composition 
of the left-hand adjoint pair in the diagram \eqref{for intro}.

Hence precomposing \eqref{to be conj} with $G^{\RA}$ from \eqref{for intro 2}, and postcomposing with $G$, gives a functorial triangle
\[
(F \funcompose F^{\RA})(a) \to a \to \twist(a) \to.
\]
Using the expressions for $F$ and $F^{\RA}$ above, the result follows.
\end{proof}

\begin{prop}\label{prop twist preserves Db}
Under the general assumptions of \ref{key assumptions}, suppose further that $f$ is crepant.  If~$\sDefAlg\in\Perf(\sTiltAlg)$, or equivalently $\twistOb\in\Perf(X)$, then
\begin{enumerate}
\item\label{prop twist preserves Db 1} $\RsHom_{\sIdeal}(\sDefAlg,\placeholder)$ preserves $\Db(\mod\sTiltAlg)$.
\item\label{prop twist preserves Db 2} $\twist$ preserves $\Db(\coh X)$.
\end{enumerate}
\end{prop}
\begin{proof}
(1) Since $\sDefAlg\in\Perf(\sTiltAlg)$, the functor $\RsHom_{\sTiltAlg}(\sDefAlg,\placeholder)$ preserves bounded complexes.  It clearly preserves coherence, and so the result follows using the two-out-of-three property for the triangles~\eqref{to be conj}. \\
(2) Again, since $\sDefAlg\in\Perf(\sTiltAlg)$, the functor $\RsHom_{\sTiltAlg}(\sDefAlg,\placeholder)$ in \eqref{for intro} preserves bounded coherent complexes, and all the other functors in \eqref{for intro} also preserve bounded coherence.  Hence $F$ and $F^{\RA}$ preserve bounded coherence, thus so does $\twist$ by \ref{funct triangle prop}, again using the two-out-of-three property.
\end{proof}

\subsection{Conditions for Equivalence}

This subsection uses the Zariski local tilting result \ref{thm zariski local tilting} to give a condition for when  $\twist$ is an equivalence globally.  Recall that $\sDefAlg$ is a \emph{Cohen--Macaulay sheaf} if it is Cohen--Macaulay at each closed point, as defined in \S\ref{section CM}.

The following notion, a translation of \ref{def sph}, will be used.

\begin{defin}\label{def geom sph}
We say that $\sDefAlg$ is \emph{$t$-relatively spherical} for a closed point $z\in Z$ if
\[
\Ext^j_{\widehat{\sTiltAlg}_z}(\widehat{\sDefAlg}_z,T)\cong\left\{  
\begin{array}{cl}
\K&\mbox{if }j=0,t\\
0&\mbox{else},
\end{array}
\right.
\]
for all simple $\widehat{\sDefAlg}_z$-modules $T$. Here $\widehat{\sDefAlg}_z$ is the completion of the stalk of $\sDefAlg$ at $z$.
\end{defin}

To set notation, choose an affine open cover $Y=\bigcup V_i$, and for any $V=V_i$ consider
\[
\RsHom_{\sTiltAlg|_V}(\sIdeal|_V,\placeholder)  \colon \D(\Mod \sTiltAlg|_V) \to \D(\Mod \sTiltAlg|_V).
\]
As $V$ is affine, say $V=\Spec R$, we may use setup~\ref{general modifying setup}, where now $\sTiltAlg|_V$ corresponds to $\Lambda$, and $\sIdeal|_V$ to the $\Lambda$-bimodule $I_{\con}$ by~\ref{affine locally it is ok}. It follows that the above functor is simply
\[
\RHom_\Lambda(I_{\con},\placeholder) \colon \D(\Mod \Lambda) \to  \D(\Mod\Lambda).
\]

The following is the main theorem of this section.

\begin{thm}\label{main sph thm}
Under the general assumptions of \ref{key assumptions}, assume that $f$ is crepant, and $\widehat{\cO}_{Y,z}$ are hypersurfaces for all closed points $z\in\nonIso$. Then the following are equivalent.
\begin{enumerate}
\item\label{main sph thm 1} $\sDefAlg$ is a Cohen--Macaulay sheaf on $Y$, and $\twistOb$ is a perfect complex on $X$.
\item\label{main sph thm 2} $\sDefAlg$ is relatively spherical for all closed points $z\in Z$.
\end{enumerate}
If these conditions hold, and they are automatic provided that $\dim Z\leq d-3$, then the functor $\twist$ is an autoequivalence of $\Db(\coh X)$.
\end{thm}
\begin{proof}
Since being Gorenstein is an open condition \cite[24.6]{Matsumura}, $Y$ is Gorenstein in a neighbourhood $N$ of $Z$.  

Conditions \eqref{main sph thm 1} and \eqref{main sph thm 2} can both be checked locally.  Since $\sDefAlg$ is supported on $Z$ by \ref{global cont thm}, and recalling from \ref{cE def} that $\twistOb$ is defined as the image of $\sDefAlg$ under an equivalence, it suffices to show that they are equivalent after restricting to $N$.  But there,  the required result is \ref{thm zariski local tilting}, which also shows that the conditions are automatic provided that $\dim Z\leq d-3$.

For the final statement, since $G$ and $G^{\RA}$ are equivalences by definition, \ref{defin geom twist}, it suffices to prove that 
\[
\RsHom_{\sTiltAlg}(\sIdeal,\placeholder)\colon \Db(\mod\sTiltAlg)\to\Db(\mod\sTiltAlg)
\] 
is an equivalence. We write $\eta$ and $\tilde{\eta}$ respectively for the counits of the following adjunctions
\begin{align*}
\placeholder\Lotimes{\sTiltAlg}\sIdeal & \dashv  \RsHom_{\sTiltAlg}(\sIdeal,\placeholder) \\
\placeholder\Lotimes{\Lambda}I_{\con} & \dashv  \RHom_{\Lambda}(I_{\con},\placeholder).
\end{align*}
Consider, for $a \in \D(\Mod\sTiltAlg)$, the distinguished triangle
\[ 
\RsHom_\sTiltAlg ( \sIdeal, a ) \Lotimes{\sTiltAlg} \sIdeal \xrightarrow{\eta_a} a \longrightarrow \operatorname{Cone}(\eta_a) \longrightarrow .
\]
This restricts to each $V=V_i$ along the inclusion $j \colon V \hookrightarrow Y$, to give 
\[ 
\RHom_{\Lambda} ( I_{\con}, j^* a ) \Lotimes{\Lambda} I_{\con} \xrightarrow{j^* \eta_a} j^* a \longrightarrow j^* \operatorname{Cone}(\eta_a) \longrightarrow .
\]
By inspection, the counit is defined locally, so $j^* \eta_a = \tilde{\eta}_{j^* a}$. 

We may take our affine open cover $Y =\bigcup V_i$ to be the union of a cover of $N\supset\nonIso$, and a cover of $Y\backslash \nonIso$.
If $V \subseteq Y\backslash Z$ then $I_{\con} = \Lambda$. If, on the other hand, $V \subseteq N$ then we have that $V$ is Gorenstein, and then $\RHom_\Lambda(I_{\con},\placeholder)$ is an equivalence by \ref{thm zariski local tilting}, condition~(\ref{thm zariski local tilting 3}). In either case, it follows that $j^* \eta_a$ is an isomorphism. Hence $j^* \operatorname{Cone}(\eta_a) = 0$ for all~$V$, so $\operatorname{Cone}(\eta_a) = 0$ for all $a$, and thence $\eta$ is an isomorphism.

The argument for the unit $\epsilon$ is similar, since it too by inspection is defined locally, and so it follows from standard adjoint functor results  that
\[
\begin{array}{c}
\begin{tikzpicture}
\node (e1) at (4,0) {$\D(\Mod\sTiltAlg)$};
\node (f1) at (8,0) {$\D(\Mod\sTiltAlg)$};
\draw[<-,transform canvas={yshift=-0.5ex}] (e1) to node [below]  {$\scriptstyle \RsHom_{\sTiltAlg}(\sIdeal,\placeholder)  $} (f1);
\draw[->,transform canvas={yshift=+0.5ex}] (e1) to  node [above] {$\scriptstyle \placeholder \Lotimes{\sTiltAlg} \sIdeal$} (f1);
\end{tikzpicture}
\end{array}
\]
is an equivalence.

By \ref{prop twist preserves Db}(\ref{prop twist preserves Db 1}) we already know that $\RsHom_{\sTiltAlg}(\sIdeal,\placeholder)$ preserves $\Db(\mod\sTiltAlg)$, so it suffices to prove that $\placeholder \Lotimes{\sTiltAlg} \sIdeal$ also has this property.

Consider $a\in\Db(\mod\sTiltAlg)$, and $a \Lotimes{\sTiltAlg} \sIdeal$.  Restricting to $V \subseteq N$, we see that 
\[
j^*(a \Lotimes{\sTiltAlg} \sIdeal)\cong j^*a\Lotimes{\Lambda} I_{\con},
\]
which belongs to $\Db(\mod\Lambda)$ since $j^*a$ does and $\placeholder\Lotimes{\Lambda} I_{\con}$ preserves $\Db(\mod\Lambda)$ by~\ref{thm zariski local tilting}. On the other hand, for $V \subseteq Y\backslash Z$, we have $I_{\con} = \Lambda$, and so $\placeholder\Lotimes{\Lambda} I_{\con}$ certainly preserves $\Db(\mod\Lambda)$.  Since we may take the cover to be finite, and the restriction of $a \Lotimes{\sTiltAlg} \sIdeal$ to each piece is bounded coherent, it follows that $a \Lotimes{\sTiltAlg} \sIdeal$ is bounded coherent. 
\end{proof}

\subsection{Application to Springer Resolutions}

The following is a corollary of \ref{main sph thm}.

\begin{cor}\label{det var thm main}
Consider the Springer resolution $X\to Y$ of the variety of singular $n\times n$ matrices. Then $\twist$ is an autoequivalence of $X$.
\end{cor}
\begin{proof}
$Y$ is a hypersurface cut out by the determinant. It is well known that $Y$ is smooth in codimension two~\cite[\S 5.1]{BLV}, and it is one of the main results of \cite{BLV} that $X$ admits a tilting bundle with trivial summand \cite[\citetype{Theorem }C]{BLV}.
\end{proof}

\section{One-Dimensional Fibre Applications}\label{1d fibre last section}

\subsection{Relative Spherical via CM Sheaves} 
The first half of this subsection is general.  If $X$ is a Gorenstein variety with tilting bundle $\cT$, necessarily $\Lambda:=\End_X(\cT)$ is an Iwanaga--Gorenstein ring, namely it is noetherian with finite injective dimension on both the left and right \cite[4.4]{KIWY}.  It follows that $\Lambda$ with its natural bimodule structure gives rise to a duality functor 
\[
\RHom_\Lambda(-,\Lambda)\colon \Db(\mod\Lambda)\to\Db(\mod\Lambda^{\op}).
\]
On the other hand, since $X$ is Gorenstein, it too has a good duality functor given by $\RsHom_X(-,\cO_X)$, and so we can consider the following diagram.
\[
\begin{array}{c}
\begin{tikzpicture}
\node (a1) at (0,0) {$\Db(\coh X)$};
\node (a2) at (4.5,0) {$\Db(\mod \Lambda)$};
\node (b1) at (0,-1.5) {$\Db(\coh X)$};
\node (b2) at (4.5,-1.5) {$\Db(\mod \Lambda^{\op})$};
\draw[->] (a1)-- node[above] {$\scriptstyle\RHom_X(\cT,-)$} node [below] {$\scriptstyle\sim$} (a2);
\draw[->] (b1) -- node[above] {$\scriptstyle\RHom_X(\cT^*,-)$} node [below] {$\scriptstyle\sim$} (b2);
\draw[->] (a1) to node[left] {$\scriptstyle \RsHom_X(-,\cO_X)$} (b1);
\draw[->] (a2) to node[right] {$\scriptstyle \RHom_\Lambda(-,\Lambda)$} (b2);
\end{tikzpicture}
\end{array}
\]

\begin{prop}\label{objectwise iso}
Suppose that $X$ is a Gorenstein variety with tilting bundle $\cT$.  Then for all $a\in\Db(\coh X)$, there is an isomorphism
\[
\RHom_X(\cT^*,\RsHom_X(a,\cO_X))\cong \RHom_\Lambda(\RHom_X(\cT,a),\Lambda).
\]
\end{prop}
\begin{proof}
Since $\RHom_X(\cT,-)$ is a dg equivalence, there is an isomorphism 
\[
\RHom_X(a,b)\cong\RHom_\Lambda(\RHom_X(\cT,a),\RHom_X(\cT,b))
\]
for all $a,b\in\Db(\coh X)$.
Hence setting $b=\cT$ gives  an isomorphism
\begin{equation}
\RHom_X(a,\cT)\cong\RHom_\Lambda(\RHom_X(\cT,a),\Lambda)\label{funct pss 1}
\end{equation}
for all $a\in\Db(\coh X)$. 

Next,  since $X$ is Gorenstein $(-)^\vee:=\RsHom_X(-,\cO_X)$ is a duality on $\Db(\coh X)$, and further since $\cT^*=\cT^\vee$, there is an isomorphism
\begin{equation}
\RHom_X(a,\cT)\cong \RHom_X(\cT^*,a^\vee) \label{funct pss 2}
\end{equation}
for all $a\in\Db(\coh X)$.  Combining \eqref{funct pss 1} and \eqref{funct pss 2} yields the result.
\end{proof}

The following is the main result of this subsection, and is specific to the setting of one-dimensional fibres.  It relates a homological condition about $\sDefAlg$ on the base $Y$ to a homological condition about $\twistOb$ at closed points of $X$. 

\begin{thm}\label{up and down CM}
In the setting of \ref{assumptions do hold}(\ref{assumptions do hold 2}), assume that $Y$ is Gorenstein in a neighbourhood of $Z$, and $f$ is crepant.  Then the following conditions are equivalent.
\begin{enumerate}
\item $\sDefAlg\in\CM_{\cS}Y$.
\item For all $y\in \nonIso$, we have $\sDefAlg_y\in\CM_{\dim \nonIso_y}\cO_{Y,y}$.
\item For all $y\in \nonIso$, and all $x\in f^{-1}(y)$, we have $\twistOb_x\in\CM_{\dim \nonIso_y+1}\cO_{X,x}$.
\end{enumerate}
In particular, if these conditions hold, then above every  $y\in\nonIso$, the exceptional locus is equidimensional of dimension $\dim Z_y+1$.
\end{thm}
\begin{proof}
(1)$\Leftrightarrow$(2) This is the definition of $\CM_{\cS}Y$, together with the fact that $\dim Z_y=\dim \sDefAlg_y$ by the contraction theorem \ref{global cont thm}.\\
(2)$\Leftrightarrow$(3) For $y\in Z$, choose a Gorenstein affine neighbourhood $V=\Spec R$ of $y$.  To simplify notation, write $\Lambda=\Lambda_V$, $k=\dim Z_y$ and $d=\dim R_y$. 

Since $f$ is crepant, $\Lambda_y$ is a maximal CM $R_y$-module \cite[4.14]{IW5}, so necessarily it is singular Calabi--Yau \cite[2.22(2)]{IW4}.  Hence
\begin{equation}
\Ext^i_{R_y}(M,R_y)\cong\Ext^i_{\Lambda_y}(M,\Lambda_y)\label{IR iso}
\end{equation}
by \cite[3.4]{IR}, for all $M\in\mod\Lambda_y$ and all $i\geq 0$.  Thus, 
\begin{align*}
\sDefAlg_y\in \CM_{k}\cO_{Y,y}
&\stackrel{\scriptsize\mbox{\ref{affine is nice}}}{\iff}
(\Lambda_{\con})_y\in\CM_{k}R_y\\
&\iff
 \RHom_{R_y}((\Lambda_{\con})_y,R_y)\mbox{ is concentrated in degree }d-k\\
 &\iff \RHom_{\Lambda_y}((\Lambda_{\con})_y,\Lambda_y)\mbox{ is concentrated in degree }d-k.
\end{align*}
Now there is a chain of isomorphisms
\begin{align*}
 \RHom_{\Lambda_y}((\Lambda_{\con})_y,\Lambda_y)&\cong
\RHom_{\Lambda_y}(\RHom_{X_y}(\locGen|_{X_y},\twistOb|_{X_y}),\Lambda_y)\\
&\cong \RHom_{X_y}(\locGen|_{X_y}^*,\RsHom_{X_y}(\twistOb|_{X_y},\cO_{X_y})),
\end{align*}
where the first follows since $\Lambda_{\con}$ corresponds across the equivalence to $\twistOb$, and the second follows by \ref{objectwise iso}  since $X_y$ is Gorenstein (since $\Spec R_y$ is), and $\locGen|_{X_y}$ is a progenerator of $\Per(X_y,R_y)$, so it is tilting on $X_y$, with endomorphism ring $\Lambda_y$ (e.g.\ \cite[4.3(2)]{IW5}).   Combining the above, we see that $\sDefAlg_y\in \CM_{k}\cO_{Y,y}$ if and only if
\begin{align*}
\RHom_{X_y}(\locGen|_{X_y}^*,\RsHom_{X_y}(\twistOb|_{X_y},\cO_{X_y})[d-k])\mbox{ is concentrated in degree }0.
\end{align*}
But now $\locGen|_{X_y}$ progenerates $\Per(X_y,R_y)$, and its dual progenerates $\PerOne(X_y,R_y)$, and hence the above condition holds if and only if
\begin{align*}
\RsHom_{X_y}(\twistOb|_{X_y},\cO_{X_y})[d-k]\in\PerOne(X_y,R_y).
\end{align*}
But by crepancy and Grothendieck duality
\begin{align*}
\Rf_*\RsHom_{X_y}(\twistOb|_{X_y},\cO_{X_y})&\cong\Rf_*\RsHom_{X_y}(\twistOb|_{X_y},f^!\cO_{R_y}) \\
& \cong
\RHom_{R_y}(\Rf_*\twistOb|_{X_y},\cO_{R_y})
\stackrel{\scriptsize\mbox{\ref{derived chase}}}{=} 
0,
\end{align*}
so it follows from \ref{Per lemma track}\eqref{Per lemma track 1} that $\sDefAlg_y\in\CM_{k}\cO_{Y,y}$ if and only if 
\[
\RsHom_{X_y}(\twistOb|_{X_y},\cO_{X_y})[d-(k+1)] \mbox{ is concentrated in degree }0,
\] 
which holds if and only if $\twistOb_x\in\CM_{k+1}\cO_{X_y,x}$ for all $x\in f^{-1}(y)$.  Since $\cO_{X_y,x}\cong \cO_{X,x}$, the result follows.

Since $\Supp_X\twistOb$ equals the exceptional locus by \ref{supp equal excp}, it follows for all $x\in f^{-1}(y)$ that $\dim\twistOb_x$ equals the dimension of the exceptional locus at the point $x$. Hence the condition $\twistOb_x\in\CM_{\dim Z_y+1}\cO_{X,x}$ forces the exceptional locus to have dimension $\dim Z_y+1$ at all points $x$ above $y$, and so the last claim follows.
\end{proof}

\begin{cor}\label{thm 1d setup updown}
In the one-dimensional fibre setting of \ref{assumptions do hold}(\ref{assumptions do hold 2}), assume that $f$ is crepant, and $\widehat{\cO}_{Y,z}$ are hypersurfaces for all closed points $z\in\nonIso$.  Then the following are equivalent.
\begin{enumerate}
\item\label{thm 1d setup updown 0} $\sDefAlg$ is relatively spherical for all closed points $z\in Z$. 
\item\label{thm 1d setup updown 1} $\sDefAlg\in\Perf(\sTiltAlg)$ and $\sDefAlg\in\CM_{\cS} Y$.
\item\label{thm 1d setup updown 2} $\twistOb\in\Perf(X)$ and $\sDefAlg\in\CM_{\cS}Y$.
\item\label{thm 1d setup updown 3} $\twistOb\in\Perf(X)$, and for all $y\in \nonIso$ and $x\in f^{-1}(y)$, we have $\twistOb_x\in\CM_{\dim \nonIso_y+1}\cO_{X,x}$.
\end{enumerate}
If these conditions hold, and they are automatic provided that $\dim Z\leq d-3$, then the functor $\twist$ is an autoequivalence of $\Db(\coh X)$.
\end{cor}
\begin{proof}
The assumptions force $Y$ to be Gorenstein in a neighbourhood of $Z$ \cite[24.6]{Matsumura}.  Recalling from \ref{cE def} that $\twistOb$ is defined as the image of $\sDefAlg$ under an equivalence, the equivalence of (\ref{thm 1d setup updown 0})--(\ref{thm 1d setup updown 3}) follows from~\ref{up and down CM}. The rest follows from \ref{main sph thm}.
\end{proof}

\begin{example}\label{poky geom example}
Consider $R=\mathbb{C}[u,v,x,y]/(uv-x^2y)$, and $Y=\Spec R$.  Then $Y$ has three crepant resolutions, sketched below.
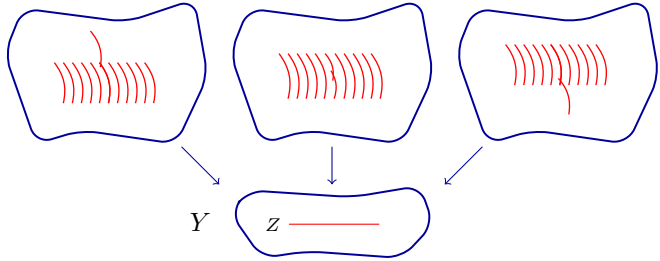
\[
\begin{array}{c}
\begin{tikzpicture}
\node at (-3,0) {\begin{tikzpicture}[scale=0.6]
\coordinate (T1) at (1.7,2.5);
\coordinate (T) at (1.9,1.8);
\coordinate (TM) at (1.92,1.7);
\coordinate (B) at (2.1,0.9);
\draw[red,line width=0.5pt] (T1) to [bend left=25] (TM);
\foreach \y in {0,0.2,0.4,...,1}{ 
\draw[line width=0.5pt,red] ($(T)+(\y,0)$) to [bend left=25] ($(B)+(\y,0)$);
\draw[line width=0.5pt,red] ($(T)+(-\y,0)$) to [bend left=25] ($(B)+(-\y,0)$);;}
\draw[color=blue!60!black,rounded corners=6pt,line width=0.75pt] (0.5,0) -- (1.5,0.3)-- (3.6,0) -- (4.3,1.5)-- (4,3.2) -- (2.5,2.7) -- (0.2,3) -- (-0.2,2)-- cycle;
\end{tikzpicture}};
\node at (0,0) {\begin{tikzpicture}[scale=0.6]
\coordinate (T) at (1.9,2);
\coordinate (TM) at (2.12-0.05,1.5-0.1);
\coordinate (BM) at (2.12-0.09,1.5+0.12);
\coordinate (B) at (2.1,1);
\draw[red,line width=0.5pt] (T) to [bend left=25] (TM);
\draw[red,line width=0.5pt] (BM) to [bend left=25] (B);
\foreach \y in {0.2,0.4,...,1}{ 
\draw[line width=0.5pt,red] ($(T)+(\y,0)+(0.02,0)$) to [bend left=25] ($(B)+(\y,0)+(0.02,0)$);
\draw[line width=0.5pt,red] ($(T)+(-\y,0)+(-0.02,0)$) to [bend left=25] ($(B)+(-\y,0)+(-0.02,0)$);;}
\draw[color=blue!60!black,rounded corners=6pt,line width=0.75pt] (0.5,0) -- (1.5,0.3)-- (3.6,0) -- (4.3,1.5)-- (4,3.2) -- (2.5,2.7) -- (0.2,3) -- (-0.2,2)-- cycle;
\end{tikzpicture}};
\node at (3,0) {\begin{tikzpicture}[scale=0.6]
\coordinate (T) at (1.9,2.2);
\coordinate (B) at (2.1,1.3);
\coordinate (BM) at (2.07,1.45);
\coordinate (B1) at (2.3,0.65);
\draw[red,line width=0.5pt] (BM) to [bend left=25] (B1);
\foreach \y in {0,0.2,0.4,...,1}{ 
\draw[line width=0.5pt,red] ($(T)+(\y,0)$) to [bend left=25] ($(B)+(\y,0)$);
\draw[line width=0.5pt,red] ($(T)+(-\y,0)$) to [bend left=25] ($(B)+(-\y,0)$);;}
\draw[color=blue!60!black,rounded corners=6pt,line width=0.75pt] (0.5,0) -- (1.5,0.3)-- (3.6,0) -- (4.3,1.5)-- (4,3.2) -- (2.5,2.7) -- (0.2,3) -- (-0.2,2)-- cycle;
\end{tikzpicture}};
\node at (0,-2) {\begin{tikzpicture}[scale=0.6]
\draw [red] (1.1,0.75) -- (3.1,0.75);
\draw[color=blue!60!black,rounded corners=6pt,line width=0.75pt] (0.5,0) -- (1.5,0.15)-- (3.6,0) -- (4.3,0.75)-- (4,1.6) -- (2.5,1.35) -- (0.2,1.5) -- (-0.2,1)-- cycle;
\node at (0.75,0.75) {${\scriptstyle \nonIso}$};

\end{tikzpicture}};
\draw[->, color=blue!60!black] (0,-1) -- (0,-1.5);
\draw[->, color=blue!60!black] (-2,-1) -- (-1.5,-1.5);
\draw[->, color=blue!60!black] (2,-1) -- (1.5,-1.5);
\node at (-1.75,-2) {$Y$};
\end{tikzpicture}
\end{array}
\]
Each gives a thickening of $\nonIso$, which we write $\sDefAlg_1$, $\sDefAlg_2$, and $\sDefAlg_3$, respectively.  Above the origin, the exceptional locus of the outer two resolutions is not equidimensional of dimension two, since in both cases there is a curve poking out of the surface.  Thus, by \ref{up and down CM} and \ref{thm 1d setup updown}, $\sDefAlg_1$ and $\sDefAlg_3$ are not relatively spherical at the origin.  Note that $\Lambda$ in \ref{poky out alg example} is derived equivalent to the left-hand resolution, and so the failure of the exceptional locus to be equidimensional  explains geometrically why $\sDefAlg_1=\Lambda_{\con}\notin\CM_{\cS}R$  in \ref{poky out alg example}.

The exceptional locus of the middle resolution is equidimensional, but this does not guarantee that $\sDefAlg_2\in\CM_{\cS}Y$.  However, to see that this indeed holds, note that the middle resolution is derived equivalent to $\Lambda_2=\End_R(R\oplus N)$, where $N=(u,x)\oplus (u,xy)$.  Since $\Omega N\cong N$, applying $\Hom_R(N,\placeholder)$ to the exact sequence
\[
0\to N\to R^4\to N\to 0
\]
gives an exact sequence
\[
0\to P_1\to P_0^{\oplus 4}\to P_1\to \sDefAlg_2\to 0.
\]
Completing the above at every closed point of $\nonIso$, we see that $\sDefAlg_2$ is $2$-relatively spherical at each closed point, so as in \ref{thm zariski local tilting}, $\pd_{\Lambda_2}\sDefAlg_2=2$ and $\sDefAlg_2\in\CM_{\cS}R$.  Thus $\twist$ is an autoequivalence on the derived category of the middle resolution.
\end{example}

\subsection{The Single Curve Fibre Case}\label{div to curve section}  In the case when there is a single curve in each fibre, and $X$ is smooth, we show here that the assumptions in \ref{thm 1d setup updown} hold.  This can arise in the setting of moduli of simple sheaves.

The following result covers both divisor-to-curve contractions and flops.

\begin{thm}\label{one curve thm}
In the one-dimensional fibre setting of \ref{assumptions do hold}(\ref{assumptions do hold 2}), suppose that $d=3$, $X$ is smooth, $Y$ is Gorenstein, and $f$ is crepant such that every reduced fibre above a closed point in $Z$ contains precisely one irreducible curve.  Then 
\begin{enumerate}
\item\label{one curve thm 1} $\sDefAlg_z\in\CM_{\cS} \cO_{Y,z}$ for all closed points $z\in Z$.
\item\label{one curve thm 2} $\sDefAlg\in\CM_{\cS} Y$.
\item\label{one curve thm 3}  $\twist$ is an autoequivalence of $X$. 
\end{enumerate}
\end{thm}
\begin{proof}
Since there is only one curve above each point $z\in\nonIso$, each $\CAz$ is local and so by \ref{loc D and hyper} every $\widehat{\cO}_{Y,z}$ is a hypersurface.  \\
(1) The assertion can be checked at the completion, thus we can assume that $Y=\Spec\mathfrak{R}$, with maximal ideal $\m$ lying in $Z$.  We just need to check that $\CA\in\CM_{\cS}\mathfrak{R}$.   But since $X$ is smooth, $\pd_\AB\CA<\infty$, and also by assumption $\dim \Spec \mathfrak{R}=3$, hence it follows by e.g.\ \cite[6.19(4)]{IW4}  that $\id_{\CA}\CA\leq 1$.  This being the case, since further $\CA$ is local, by Ramras \cite[2.15]{Ramras} there is a chain of equalities
\[
\depth_{\mathfrak{R}} \CA = \dim_{\mathfrak{R}} \CA = \id_{\CA}\CA
\]
and hence $\CA\in\CM_{\cS}\mathfrak{R}$.\\
(2) Since the support of $\sDefAlg$ equals $Z$ by \ref{global cont thm}, this is an immediate corollary of \eqref{one curve thm 1}.\\
(3) Since $X$ is smooth, automatically $\twistOb\in\Perf(X)$.  Hence the result follows by combining~\eqref{one curve thm 2} and~\ref{thm 1d setup updown}(\ref{thm 1d setup updown 2}).
\end{proof}

\end{document}